\documentclass[11pt,reqno]{amsart}

\usepackage{amsmath,amssymb,amsthm,mathtools,mathrsfs}
\usepackage[a4paper,margin=30mm]{geometry}
\usepackage[hidelinks]{hyperref}
\usepackage{verbatim}

\numberwithin{equation}{section}

\newtheorem{theorem}{Theorem}[section]
\newtheorem{proposition}[theorem]{Proposition}
\newtheorem{lemma}[theorem]{Lemma}
\newtheorem{corollary}[theorem]{Corollary}
\newtheorem{remark}[theorem]{Remark}
\newtheorem{claim}[theorem]{Claim}
\newtheorem*{classicalBL}{Classical Br\'ezis--Lieb lemma}
\newcommand{\R}{\mathbb R}
\newcommand{\Hhalf}{\R^n_+}
\newcommand{\bdH}{\partial\R^n_+}
\newcommand{\ST}{S_{\mathrm T}}
\newcommand{\MT}{\mathcal M_{\mathrm T}}
\newcommand{\NT}{\mathcal N_{\mathrm T}}
\newcommand{\Dp}{\dot W^{1,p}}
\newcommand{\PT}{\mathcal P_{\mathrm T}}

\newcommand{\dist}{\operatorname{dist}}

\title[Stability and compactness for the Sobolev trace inequality]
{Sharp One-bubble Critical-Point Stability and Global Compactness for the
Sobolev Trace Inequality}
\author{Bao Yu}
\thanks{
Bao Yu:
School of Mathematical Sciences,
University of Science and Technology of China,
Hefei, Anhui Province, P.~R.~China, 230006.
Email: \texttt{baoyu1@mail.ustc.edu.cn}.
}

\author{Yang Zhou}
\thanks{
Yang Zhou:
School of Mathematical Sciences,
University of Science and Technology of China,
Hefei, Anhui Province, P.~R.~China, 230006.
Email: \texttt{zy19700816@mail.ustc.edu.cn}.
}
\date{}

\begin{document}
\begin{abstract}
Let $n\ge3$ and $1<p<n$.  We first prove the local trace analogue of
the sharp one-bubble critical-point stability theorem of Liu and
Zhang~\cite{LiuZhang2025}: near a positive trace-bubble, the
Euler--Lagrange residual controls the gradient distance to the
normalized trace-bubble manifold with the sharp power
$\max\{1,p-1\}$. Then, we establish a Struwe-type compactness theorem for
the critical trace functional, which gives the trace counterpart of
the Mercuri--Willem decomposition~\cite{MercuriWillem2010}.
Combining Struwe-type compactness with the local stability estimate yields
a sharp quantitative one-bubble critical-point stability theorem.
\end{abstract}

\maketitle

\section{Introduction and main results}

\subsection{Background and motivation}

Given $n\ge 3$ and $1<p<n$, we denote by
$\dot W^{1,p}(\mathbb R^n)$ the completion of
$C_c^\infty(\mathbb R^n)$ under the seminorm
\[
\|u\|_{\dot W^{1,p}(\mathbb R^n)}
:=
\left(
\int_{\mathbb R^n}|\nabla u|^p\,dx
\right)^{\frac1p}.
\]
Let
\[
 p^*:=\frac{np}{n-p},
 \qquad
 p_*:=\frac{(n-1)p}{n-p},
 \qquad
 p':=\frac{p}{p-1}.
\]
Additionally, denote by $W^{-1,p'}(\mathbb R^n)$ the dual space of
$\dot W^{1,p}(\mathbb R^n)$.
The sharp Sobolev inequality states that
\begin{equation}
 S(n,p)\|u\|_{L^{p^*}(\R^n)}
 \le \|\nabla u\|_{L^p(\R^n)},
 \qquad u\in\dot W^{1,p}(\R^n).
 \label{eq:sobolev-background}
\end{equation}
Aubin and Talenti~\cite{Aubin1976,Talenti1976} identified the sharp
constant and showed that all extremals form the $(n+2)$-dimensional
manifold
\begin{equation}
 \mathcal M_{\mathrm S}
 :=\{aT[z,\lambda]:a\ne0,\ \lambda>0,\ z\in\R^n\},
 \qquad
 T[z,\lambda](x)
 :=c_{n,p}
 \left(\frac{\lambda}
 {1+\lambda^{p'}|x-z|^{p'}}\right)^{\frac{n-p}{p}}.
 \label{eq:talenti-family-background}
\end{equation}
Here $c_{n,p}>0$ is chosen so that the normalized Talenti bubbles
solve the corresponding Euler-Lagrange equation of \eqref{eq:sobolev-background}, 
\begin{equation}
 -\operatorname{div}(|\nabla T|^{p-2}\nabla T)
 =T^{p^*-1}
 \qquad\text{in }\R^n.
 \label{eq:sobolev-critical-equation-background}
\end{equation}

Once the minimizers of \eqref{eq:sobolev-background} and all positive
solutions of \eqref{eq:sobolev-critical-equation-background} are known,
two natural stability questions arise.  In the functional setting, does
the deficit control the distance to the extremal manifold
$\mathcal M_{\mathrm S}$?  Similarly, in the PDE setting, does being an
approximate positive solution of
\eqref{eq:sobolev-critical-equation-background} imply being close to a
Talenti bubble?

The question on the stability of functional inequalities was first raised by Br\'ezis and Lieb~\cite{BrezisLieb1985} for $p=2$, and it was settled by Bianchi and Egnell~\cite{BianchiEgnell1991}.
More precisely, for $u\ne0$ set
\[
 \delta_{\mathrm S}(u)
 :=\frac{\|\nabla u\|_{L^p(\R^n)}}
          {\|u\|_{L^{p^*}(\R^n)}}-S(n,p),
 \qquad
 d_{\mathrm S}(u)
 :=\inf_{v\in\mathcal M_{\mathrm S}}
 \frac{\|\nabla(u-v)\|_{L^p(\R^n)}}
      {\|\nabla u\|_{L^p(\R^n)}}.
\]
When $p=2$, Bianchi and Egnell~\cite{BianchiEgnell1991} proved that
there is $C_{\mathrm{BE}}(n)>0$ such that
\begin{equation}
 \delta_{\mathrm S}(u)
 \ge C_{\mathrm{BE}}(n)d_{\mathrm S}(u)^2,
 \label{eq:bianchi-egnell-background}
\end{equation}
and both the gradient distance and the exponent $2$ are optimal. However, Bianchi--Egnell's method relies heavily on the Hilbert-space
structure available at \(p=2\) and does not extend directly to
\(p\ne2\). For
all $1<p<n$, the sharp gradient stability was finally proved by Figalli and
Zhang~\cite{FigalliZhang2022}
\begin{equation}
 \delta_{\mathrm S}(u)
 \ge c(n,p)d_{\mathrm S}(u)^{\max\{2,p\}}.
 \label{eq:sobolev-functional-stability-background}
\end{equation}
Their proof combines delicate nonlinear Taylor-remainder estimates with
a perturbed spectral-gap inequality around a Talenti bubble.

As for the stability of critical points, the qualitative starting point is Struwe's global compactness
theorem~\cite{Struwe1984}, extended to the $p$-Laplacian by Mercuri and
Willem~\cite{MercuriWillem2010}. Let $2^*=2n/(n-2)$ and define
\[
 \Gamma_{\mathrm S}(u)
 :=\bigl\|\Delta u+|u|^{2^*-2}u\bigr\|_{\dot H^{-1}(\R^n)}.
\]
Ciraolo, Figalli, and Maggi~\cite{CiraoloFigalliMaggi2018} obtained the
first sharp quantitative one-bubble estimate: for suitable
$a_1>0$, $z\in\R^n$, and $\lambda>0$,
\begin{equation}
 \|\nabla u-a_1\nabla T[z,\lambda]\|_{L^2(\R^n)}
 \le C\Gamma_{\mathrm S}(u).
 \label{eq:cfm-one-bubble-background}
\end{equation}
Later, Figalli and Glaudo~\cite{FigalliGlaudo2020} proved the corresponding
linear multi-bubble estimate, for suitable $a_i>0$, $z_i\in\R^n$,
and $\lambda_i>0$,
\begin{equation}
 \left\|\nabla u-\sum_{i=1}^{\nu}
 a_i\nabla T[z_i,\lambda_i]\right\|_{L^2(\R^n)}
 \le C\Gamma_{\mathrm S}(u),
 \qquad 3\le n\le5,
 \label{eq:figalli-glaudo-background}
\end{equation}
and showed that linear control fails in higher dimensions.  Deng, Sun,
and Wei~\cite{DengSunWei2025} then established the sharp rates
\begin{equation}
 \left\|\nabla u-\sum_{i=1}^{\nu}
 a_i\nabla T[z_i,\lambda_i]\right\|_{L^2(\R^n)}
 \le C
 \begin{cases}
  \Gamma_{\mathrm S}(u)|\log\Gamma_{\mathrm S}(u)|^{1/2},&n=6,\\
  \Gamma_{\mathrm S}(u)^{\frac{n+2}{2(n-2)}},&n\ge7.
 \end{cases}
 \label{eq:deng-sun-wei-background}
\end{equation}
For general $1<p<n$, write
\[
 \mathcal P_{\mathrm S}(u)
 :=-\operatorname{div}(|\nabla u|^{p-2}\nabla u)
   -|u|^{p^*-2}u,
 \qquad
 \mathcal N_{\mathrm S}:=\{T[z,\lambda]:\lambda>0, z\in\R^n\}.
\]
Liu and Zhang~\cite{LiuZhang2025} proved the sharp local one-bubble
estimate: if $\inf_{T\in\mathcal N_{\mathrm S}}
 \|\nabla(u-T)\|_{L^p(\R^n)}$ is small enough,
\begin{equation}
 \inf_{T\in\mathcal N_{\mathrm S}}
 \|\nabla(u-T)\|_{L^p(\R^n)}^{\max\{1,p-1\}}
 \le C
 \|\mathcal P_{\mathrm S}(u)\|_{W^{-1,p'}(\R^n)}.
 \label{eq:sobolev-critical-stability-background}
\end{equation}
Liu and Zhang then combine this local estimate with the compactness
theorem of Mercuri and Willem \cite{MercuriWillem2010} to show that,
for any nonnegative function
\(u \in \dot{W}^{1,p}(\mathbb{R}^n)\) such that
\[
\frac{1}{2}S(n,p)^n
\le
\int_{\mathbb{R}^n} |Du|^p\,dx
\le
\frac{3}{2}S(n,p)^n,
\]
there exists \(C=C(n,p)>0\) such that
\[
\inf_{T\in\mathcal N_{\mathrm S}}\|\nabla(u-T)\|_{L^p(\mathbb{R}^n)}^{\max\{1,p-1\}}
\le
C\|\mathcal P_{\mathrm S}(u)\|_{W^{-1,p'}(\mathbb{R}^n)}.
\]

Spectral-gap estimates play a central role in the stability
results described above. Their derivation often exploits the radial
symmetry of both the reference bubble and the underlying geometry,
which allows the linearized eigenvalue problem to be reduced, by
separation of variables, to a family of ordinary differential
equations amenable to Sturm--Liouville theory.
A different approach, based on a suitable $P$-function, was developed
by Ciraolo and Gatti~\cite{CiraoloGatti2026}. Although the resulting
quantitative rate is not optimal, this method is more
robust with respect to anisotropy. More recently, Antonini, Ciraolo,
and Gatti~\cite{AntoniniCiraoloGatti2026} established classification
and global decomposition results for the anisotropic critical
$p$-Laplace equation and obtained a quantitative one-bubble critical-point stability estimate.

On the trace side, the sharp inequality and its extremals were
identified by Escobar, Nazaret and Maggi-Neumayer~\cite{Escobar1988,Nazaret2006,MN17}.  For
$u$ with nonzero trace, introduce
\[
 \delta_{\mathrm T}(u)
 :=\frac{\|\nabla u\|_{L^p(\Hhalf)}}
          {\|u\|_{L^{p_*}(\bdH)}}-\ST(n,p),
 \qquad
 d_{\mathrm T}(u)
 :=\inf_{V\in\MT}
 \frac{\|\nabla(u-V)\|_{L^p(\Hhalf)}}
      {\|\nabla u\|_{L^p(\Hhalf)}}.
\]
When $p=2$, Ho~\cite{Ho2022} proved the trace analogue of
the Bianchi--Egnell theorem,
\begin{equation}
 \delta_{\mathrm T}(u)
 \ge c(n)d_{\mathrm T}(u)^2.
 \label{eq:ho-trace-stability-background}
\end{equation}
More recently, Zhang, Zhou, and Zou~\cite{ZhangZhouZou2025} established
the fractional trace estimate
\begin{equation}
 \begin{split}
 C_{\mathrm{BE}}(n,m,\alpha)
 \inf_{v\in\mathcal M_{n,m,\alpha}}
 \|f-v\|_{D_\alpha(\R^n)}^2
 &\le \|f\|_{D_\alpha(\R^n)}^2\\
 &\quad-S(n,m,\alpha)
 \|\tau_m f\|_{L^q(\R^{n-m})}^2,
 \end{split}
 \label{eq:zzz-functional-trace-background}
\end{equation}
where $0\le m<n$, $m/2<\alpha<n/2$, and
$q=2(n-m)/(n-2\alpha)$. A key ingredient is the reduction principle in
\cite[Theorem~3.3]{ZhangZhouZou2025}, which reduces the
stability of the fractional trace inequality to a stability estimate for the fractional
Sobolev inequality. They also obtained a sharp
quantitative profile decomposition for the Escobar equation.  In weak
form, its residual is
\[
 \Gamma_{\mathrm E}(u)
 :=\sup_{\|\nabla\varphi\|_{L^2(\Hhalf)}\le1}
 \left|
 \int_{\Hhalf}\nabla u\cdot\nabla\varphi\,dx
 -\int_{\bdH}|u|^{\frac{2}{n-2}}u\varphi\,dy
 \right|.
\]
If $d_{\mathrm E}^\nu(u)$ denotes the gradient distance to a sum of
$\nu$ weakly interacting Escobar bubbles, their estimate states
\begin{equation}
 d_{\mathrm E}^\nu(u)\le C(n,\nu)\Gamma_{\mathrm E}(u),
 \qquad
 \begin{cases}
  \nu=1,&n\ge3,\\
  \nu\ge2,&n=3.
 \end{cases}
 \label{eq:zzz-critical-trace-background}
\end{equation}

For the classical trace inequality and every $1<p<n$, Ma, Zhang and
the second author~\cite{TraceGradientStability} recently proved the sharp gradient
stability estimate
\begin{equation}
 \delta_{\mathrm T}(u)
 \ge c(n,p)d_{\mathrm T}(u)^{\max\{2,p\}}.
 \label{eq:mzz-trace-gradient-background}
\end{equation}
A central part of their proof is the spectral gap estimate, since the lack of radial symmetry of weight and domain creates new spectral
difficulties that prevent a direct adaptation of the Sobolev argument.
Moreover, the second author~\cite{Zhou2024} classified all positive finite-energy
solutions of the associated critical trace equation.

\subsection{Problem setup and main results}
The purpose of this paper is to establish the sharp quantitative one-bubble critical-point stability theorem, which extends the whole-space theorem of
\cite{LiuZhang2025} to the trace setting, and the \(p=2\) trace
theorem of \cite{ZhangZhouZou2025} to the full range \(1<p<n\).

Recall that the sharp trace inequality~\cite{Escobar1988,Nazaret2006} states
\begin{equation}
 \ST(n,p)\|u\|_{L^{p_*}(\bdH)}
 \le \|\nabla u\|_{L^p(\Hhalf)},
 \qquad u\in\Dp(\Hhalf).
 \label{eq:sharp-trace-cp}
\end{equation}
Write
\[
 U(y,t):=\bigl(|y|^2+(t+1)^2\bigr)^{-\frac{n-p}{2(p-1)}}
\]
and, for $\lambda>0$ and $\xi\in\R^{n-1}$, set
\begin{equation}
 (G_{\lambda,\xi}f)(y,t)
 :=\lambda^{-\frac{n-p}{p}}
 f\left(\frac{y-\xi}{\lambda},\frac{t}{\lambda}\right),
 \qquad
 U[\lambda,\xi]:=G_{\lambda,\xi}U.
 \label{eq:trace-dislocation}
\end{equation}
The maps $G_{\lambda,\xi}$ are isometries both on
$\Dp(\Hhalf)$ and on $L^{p_*}(\bdH)$.  The classification of the trace
extremals
\cite{Nazaret2006,MN17} gives the full extremal manifold and its positive component
\begin{equation}
 \begin{aligned}
 \MT
 &:=\{aU[\lambda,\xi]:a\ne0,\ \lambda>0,\ \xi\in\R^{n-1}\},\\
 \MT^+
 &:=\{aU[\lambda,\xi]:a>0,\ \lambda>0,\ \xi\in\R^{n-1}\}.
 \end{aligned}
 \label{eq:positive-trace-manifold-cp}
\end{equation}
For $v\in\MT^+$ set
\begin{equation}
 \Lambda_v
 :=\ST^p\|v\|_{L^{p_*}(\bdH)}^{p-p_*}.
 \label{eq:Lambda-v-cp}
\end{equation}
Then
\begin{equation}
 \int_{\Hhalf}|\nabla v|^{p-2}\nabla v\cdot\nabla\psi\,dx
 =\Lambda_v\int_{\bdH}v^{p_*-1}\psi\,dy
 \label{eq:EL-v-cp}
\end{equation}
for every $\psi\in\Dp(\Hhalf)$.

Put
$\Lambda_U:=\ST^p\|U\|_{L^{p_*}(\bdH)}^{p-p_*}$ and
\[
 a_{\mathrm T}:=\Lambda_U^{1/(p_*-p)},
 \qquad
 W[\lambda,\xi]:=a_{\mathrm T}U[\lambda,\xi].
\]
Thus $\Lambda_{W[\lambda,\xi]}=1$ and
\begin{equation}
 A_{\mathrm T}
 :=\|W[\lambda,\xi]\|_{L^{p_*}(\bdH)}^{p_*}
 =\|\nabla W[\lambda,\xi]\|_{L^{p}(\Hhalf)}^p
 =\ST^{\frac{(n-1)p}{p-1}},
 \qquad
 E_{\mathrm T}:=\frac{p-1}{p(n-1)}A_{\mathrm T}.
 \label{eq:trace-bubble-energy}
\end{equation}
Moreover, the normalized solution manifold
\begin{equation}
 \NT:=\{W[\lambda,\xi]:\lambda>0,\ \xi\in\R^{n-1}\}
 \label{eq:normalized-solution-manifold}
\end{equation}
is precisely the positive trace-bubble family solving
\begin{equation}
 \begin{cases}
  \operatorname{div}(|\nabla W|^{p-2}\nabla W)=0&\text{in }\Hhalf,\\
  |\nabla W|^{p-2}\nabla W\cdot\nu=W^{p_*-1}&\text{on }\bdH.
 \end{cases}
 \label{eq:normalized-trace-equation}
\end{equation}

We denote by $W^{-1,p'}(\Hhalf)$ the dual space of $\Dp(\Hhalf)$.
Define the Euler--Lagrange residual
$\mathcal P_{\mathrm T}(u)\in W^{-1,p'}(\Hhalf)$ by
\begin{equation}
 \begin{split}
 \langle\mathcal P_{\mathrm T}(u),\psi\rangle
 &:={}
 \int_{\Hhalf}|\nabla u|^{p-2}\nabla u\cdot\nabla\psi\,dx\\
 &\quad-
 \int_{\bdH}|u|^{p_*-2}u\,\psi\,dy,
 \end{split}
 \label{eq:trace-residual-definition}
\end{equation}
and
\begin{equation}
 \|\mathcal P_{\mathrm T}(u)\|_{W^{-1,p'}(\Hhalf)}
 :=\sup_{\substack{\psi\in\Dp(\Hhalf)\\
                   \|\nabla\psi\|_{L^p(\Hhalf)}\le1}}
 |\langle\mathcal P_{\mathrm T}(u),\psi\rangle|.
 \label{eq:trace-residual-norm}
\end{equation}
The boundary integral in \eqref{eq:trace-residual-definition} is
well defined by \eqref{eq:sharp-trace-cp}.
We write
\[
 \dist_{\Dp}(u,\MT)
 :=\inf_{z\in\MT}\|\nabla(u-z)\|_{L^p(\Hhalf)}.
\]
Similarly, set
\[
 \dist_{\Dp}(u,\NT)
 :=\inf_{W\in\NT}\|\nabla(u-W)\|_{L^p(\Hhalf)}.
\]

Following ideas of Liu and Zhang in
\cite{LiuZhang2025} and using the spectral gap estimate in
\cite[Corollary~2.6]{TraceGradientStability}, we derive the following
optimal quantitative stability estimate.
\begin{theorem}[Local one-bubble critical-point stability]
\label{thm:trace-critical-stability}
There exist $\delta=\delta(n,p)>0$ and $C=C(n,p)>0$ with the
following property.  Suppose that $u\in\Dp(\Hhalf)$ and
\begin{equation}
 \inf_{W\in\NT}
 \|\nabla(u-W)\|_{L^p(\Hhalf)}\le\delta.
 \label{eq:one-bubble-closeness}
\end{equation}
Then there exists $v\in\NT$ such that, with $h:=u-v$,
\begin{equation}
 \|\nabla h\|_{L^p(\Hhalf)}^{\max\{1,p-1\}}
 \le
 C\|\mathcal P_{\mathrm T}(u)\|_{W^{-1,p'}(\Hhalf)}.
 \label{eq:critical-point-stability-main}
\end{equation}
Consequently,
\begin{equation}
 \dist_{\Dp}(u,\NT)^{\max\{1,p-1\}}
 \le C\|\mathcal P_{\mathrm T}(u)\|_{W^{-1,p'}(\Hhalf)}.
 \label{eq:critical-point-distance-form}
\end{equation}
The exponent $\max\{1,p-1\}$ is optimal.
\end{theorem}
\begin{remark}
  As in \cite{LiuZhang2025}, the
sharpness of the exponent $1$ when $p<2$ in Theorem \ref{thm:trace-critical-stability} follows from a smooth perturbation
of $v$, and that of the exponent $p-1$ when $p\ge 2$ follows similarly as the one in \cite{ZhouZouCKN}.
\end{remark}

Arguing as in \cite{MercuriWillem2010}, we next prove the following Struwe-type global compactness theorem that removes the
one-bubble closeness assumption at the qualitative level.  For
$s\in\R$ write $s_-:=\max\{-s,0\}$, and define
\begin{equation}
 J_+(u)
 :=\frac1p\int_{\Hhalf}|\nabla u|^p\,dx
 -\frac1{p_*}\int_{\bdH}(u_+)^{p_*}\,dy.
 \label{eq:positive-functional}
\end{equation}

\begin{theorem}[Global compactness on the half-space]
\label{thm:global-compactness}
Let $(u_k)\subset\Dp(\Hhalf)$ be bounded and suppose that
\begin{equation}
 \|\PT(u_k)\|_{W^{-1,p'}(\Hhalf)}\longrightarrow0,
 \qquad
 \|(u_k)_-\|_{L^{p_*}(\bdH)}\longrightarrow0.
 \label{eq:global-PS-assumptions}
\end{equation}
Then, after passing to a subsequence, there are an integer $m\ge0$ and
parameters $\lambda_j^{(k)}>0$, $\xi_j^{(k)}\in\R^{n-1}$,
$1\le j\le m$, such that
\begin{equation}
 \left\|\nabla\left(
 u_k-\sum_{j=1}^mW[\lambda_j^{(k)},\xi_j^{(k)}]
 \right)\right\|_{L^p(\Hhalf)}\longrightarrow0.
 \label{eq:global-decomposition}
\end{equation}
For $i\ne j$ the parameters are asymptotically orthogonal:
\begin{equation}
 \frac{\lambda_i^{(k)}}{\lambda_j^{(k)}}
 +\frac{\lambda_j^{(k)}}{\lambda_i^{(k)}}
 +\frac{|\xi_i^{(k)}-\xi_j^{(k)}|^2}
        {\lambda_i^{(k)}\lambda_j^{(k)}}
 \longrightarrow\infty.
 \label{eq:global-parameter-orthogonality}
\end{equation}
Moreover,
\begin{align}
 \int_{\Hhalf}|\nabla u_k|^p\,dx
 &\longrightarrow mA_{\mathrm T},
 \label{eq:global-gradient-quantization}\\
 J_+(u_k)&\longrightarrow mE_{\mathrm T}.
 \label{eq:global-energy-quantization}
\end{align}
\end{theorem}

As a corollary of
Theorem~\ref{thm:global-compactness}, we obtain the trace counterpart of
\cite[Theorem~1.1]{LiuZhang2025}.

\begin{corollary}
\label{thm:trace-nu-bubble-compactness}
Let $\nu\ge1$ be an integer, and let
$(u_k)\subset\Dp(\Hhalf)$ be a sequence of nonnegative functions such
that
\begin{equation}
 \left(\nu-\frac12\right)A_{\mathrm T}
 \le\int_{\Hhalf}|\nabla u_k|^p\,dx
 \le\left(\nu+\frac12\right)A_{\mathrm T}
 \label{eq:trace-nu-energy-window}
\end{equation}
and
\[
 \|\PT(u_k)\|_{W^{-1,p'}(\Hhalf)}\longrightarrow0.
\]
Then there are $\nu$-tuples of parameters
$\lambda_j^{(k)}>0$ and $\xi_j^{(k)}\in\R^{n-1}$ such that
\begin{equation}
 \left\|\nabla\left(
 u_k-\sum_{j=1}^{\nu}
 W[\lambda_j^{(k)},\xi_j^{(k)}]
 \right)\right\|_{L^p(\Hhalf)}\longrightarrow0.
 \label{eq:trace-nu-bubble-decomposition}
\end{equation}
\end{corollary}

Combining the case $\nu=1$ with
Theorem~\ref{thm:trace-critical-stability} gives the trace analogue of
\cite[Corollary~1.3]{LiuZhang2025}.

\begin{corollary}[Global one-bubble critical-point stability]
\label{cor:global-one-bubble-critical-stability}
Let $u\in\Dp(\Hhalf)$ be nonnegative and satisfy
\begin{equation}
 \frac12A_{\mathrm T}
 \le\int_{\Hhalf}|\nabla u|^p\,dx
 \le\frac32A_{\mathrm T}.
 \label{eq:global-one-bubble-energy-window}
\end{equation}
Then there exist $v\in\NT$ and $C=C(n,p)>0$ such that
\begin{equation}
 \|\nabla(u-v)\|_{L^p(\Hhalf)}^{\max\{1,p-1\}}
 \le C\|\PT(u)\|_{W^{-1,p'}(\Hhalf)}.
 \label{eq:global-one-bubble-critical-stability}
\end{equation}
\end{corollary}

The proof of Theorem~\ref{thm:trace-critical-stability} follows the strategy of \cite{LiuZhang2025}, but it includes
non-trivial modifications. Specifically, instead of giving a delicate lower bound of $\mathscr M_p(X,Y):=\bigl(|X+Y|^{p-2}(X+Y)-|X|^{p-2}X\bigr)\cdot Y$ and relying on
\cite[Lemma 2.1]{LiuZhang2025}, we work directly with the exact remainder $\mathscr M_p$ as in \cite{TraceGradientStability}, which simplifies the proof of the disturbed spectral gap estimate; see Proposition~\ref{prop:disturbed-trace-gap-cp}.

This paper is organized as follows.
Section~2 establishes pointwise estimates for the nonlinear remainders and uses the spectral gap in
\cite{TraceGradientStability} to prove the disturbed spectral gap needed
for Theorem~\ref{thm:trace-critical-stability}. In Section~3, we prove
the local sharp stability theorem, Theorem~\ref{thm:trace-critical-stability}.
In Section~4, following the method in \cite{MercuriWillem2010}, we prove the global compactness theorem and its two
consequences.

\section{The disturbed spectral gap for the critical-point remainder}

\subsection{Pointwise remainders and the linear trace gap}

For $X,Y\in\R^n$, $a\ge0$, $b\in\R$, and $q>1$, define
\begin{align}
 \mathscr M_p(X,Y)
 &:=
 \bigl(|X+Y|^{p-2}(X+Y)-|X|^{p-2}X\bigr)\cdot Y,
 \label{eq:M-critical-definition}\\
 \mathscr N_q(a,b)
 &:=
 \bigl(|a+b|^{q-2}(a+b)-a^{q-1}\bigr)b.
 \label{eq:N-critical-definition}
\end{align}
These are the exact bulk and boundary remainders obtained after testing
the Euler--Lagrange equation by the perturbation.

\begin{lemma}[Pointwise estimates]
\label{lem:pointwise-critical-remainders}
There are $c_p,C_p>0$ such that
\begin{equation}
 c_p(|X|+|Y|)^{p-2}|Y|^2
 \le\mathscr M_p(X,Y)
 \le C_p(|X|+|Y|)^{p-2}|Y|^2.
 \label{eq:M-monotonicity-comparison}
\end{equation}
In particular, if $1<p<2$, then
\begin{equation}
 \mathscr M_p(X,Y)
 \simeq_p\min\{|Y|^p,|X|^{p-2}|Y|^2\},
 \label{eq:M-min-form-critical}
\end{equation}
whereas, if $p\ge2$, then
\begin{equation}
 \mathscr M_p(X,Y)
 \ge c_p\bigl(|X|^{p-2}|Y|^2+|Y|^p\bigr).
 \label{eq:M-p-ge-two-critical}
\end{equation}

For every $\kappa>0$ there are constants $C_0,C_1>0$, depending only
on $q$ and $\kappa$, such that
\begin{align}
 \mathscr N_q(a,b)
 &\le(q-1+\kappa)
 \frac{(a+C_0|b|)^q}{a^2+b^2}|b|^2,
 &&1<q\le2,
 \label{eq:N-q-le-two-critical}\\
 \mathscr N_q(a,b)
 &\le(q-1+\kappa)a^{q-2}|b|^2+C_1|b|^q,
 &&q>2.
 \label{eq:N-q-gt-two-critical}
\end{align}
At $(a,b)=(0,0)$ the quotient in
\eqref{eq:N-q-le-two-critical} is understood to be zero.
\end{lemma}

\begin{proof}
Formula \eqref{eq:M-monotonicity-comparison} is the standard strict
monotonicity estimate for $X\mapsto|X|^{p-2}X$.  Splitting into
$|Y|\le |X|/2$ and $|Y|>|X|/2$ gives
\eqref{eq:M-min-form-critical} and
\eqref{eq:M-p-ge-two-critical}.

For the scalar estimates, the case $a=0$ follows directly from
$\mathscr N_q(0,b)=|b|^q$; for $a>0$, homogeneity reduces the proof to
$a=1$.
On $|b|\le b_0$, Taylor expansion gives
\[
 (|1+b|^{q-2}(1+b)-1)b=(q-1)b^2+o(b^2).
\]
On the complementary region, the right-hand sides of
\eqref{eq:N-q-le-two-critical} and
\eqref{eq:N-q-gt-two-critical} absorb the expression after increasing
$C_0$ or $C_1$.  This is the scalar argument used by Liu--Zhang
\cite[Lemma~2.2]{LiuZhang2025}, with $p^*$ replaced by the trace
exponent $q$.
\end{proof}

For $v\in\MT^+$ define
\begin{equation}
 \mathcal A_v
 :=|\nabla v|^{p-2}
 \left(I+(p-2)\frac{\nabla v}{|\nabla v|}
 \otimes\frac{\nabla v}{|\nabla v|}\right).
 \label{eq:A-v-critical}
\end{equation}
The replacement for \cite[Lemma~3.3]{LiuZhang2025} is the following
direct consequence of
\cite[Corollary~2.6]{TraceGradientStability}.

\begin{proposition}[Linear trace spectral gap]
\label{prop:linear-trace-gap-cp}
There is $\lambda_{\mathrm T}=\lambda_{\mathrm T}(n,p)>0$ such that
every $\phi\in\dot W^{1,2}(\Hhalf;|\nabla U|^{p-2})$ satisfying
\begin{equation}
 \int_{\bdH}U^{p_*-2}\phi Z\,dy=0
 \qquad\forall Z\in T_U\MT,
 \label{eq:linear-orthogonality-cp}
\end{equation}
obeys
\begin{equation}
 \int_{\Hhalf}\mathcal A_U\nabla\phi\cdot\nabla\phi\,dx
 \ge
 \left[(p_*-1)\ST^p+2\lambda_{\mathrm T}\right]
 \|U\|_{L^{p_*}(\bdH)}^{\,p-p_*}
 \int_{\bdH}U^{p_*-2}|\phi|^2\,dy.
 \label{eq:linear-trace-gap-cp}
\end{equation}
By critical dilations, tangential translations, and multiplication, the
corresponding estimate at $v\in\MT^+$ has coefficient
\[
 \left[(p_*-1)\ST^p+2\lambda_{\mathrm T}\right]
 \|v\|_{L^{p_*}(\bdH)}^{\,p-p_*}.
\]
\end{proposition}

\subsection{The disturbed spectral gap}

Set
\begin{equation}
 \mathcal H_v[h]
 :=\int_{\Hhalf}\mathscr M_p(\nabla v,\nabla h)\,dx.
 \label{eq:H-v-critical}
\end{equation}
Here and below, $h\perp T_v\MT$ means
\[
 \int_{\bdH}v^{p_*-2}hZ\,dy=0
 \qquad\text{for every }Z\in T_v\MT.
\]
The next proposition is the trace version of
\cite[Proposition~3.4]{LiuZhang2025}.  Its proof follows the nonlinear
compactness scheme of \cite[Section~4]{TraceGradientStability}; the
only change is that the
energy remainder there is replaced by the monotonicity remainder
\eqref{eq:M-critical-definition}.

\begin{proposition}[Disturbed trace spectral gap]
\label{prop:disturbed-trace-gap-cp}
If $p_*\le2$, then for every fixed $C_0\ge1$ there is
$\eta_0=\eta_0(n,p,C_0)>0$ such that, whenever
$v\in\MT^+$, $h\in\Dp(\Hhalf)$,
$h\perp T_v\MT$,
\begin{equation}
 \frac12\ST^{\frac{n-p}{p-1}}
 \le \|v\|_{L^{p_*}(\bdH)}
 \le \frac32\ST^{\frac{n-p}{p-1}},
 \label{eq:disturbed-gap-v-normalization}
\end{equation}
and
$\|\nabla h\|_p\le\eta_0$, one has
\begin{equation}
 \mathcal H_v[h]
 \ge
 \left[(p_*-1)\ST^p+\lambda_{\mathrm T}\right]
 \|v\|_{L^{p_*}(\bdH)}^{\,p-p_*}
 \int_{\bdH}
 \frac{(v+C_0|h|)^{p_*}}{v^2+|h|^2}|h|^2\,dy.
 \label{eq:disturbed-gap-q-le-two}
\end{equation}

If $p_*>2$, there is $\eta_1=\eta_1(n,p)>0$ such that every
$v\in\MT^+$ and $h\in\Dp(\Hhalf)$ satisfying
$h\perp T_v\MT$, \eqref{eq:disturbed-gap-v-normalization}, and
$\|\nabla h\|_p\le\eta_1$ obey
\begin{equation}
 \mathcal H_v[h]
 \ge
 \left[(p_*-1)\ST^p+\lambda_{\mathrm T}\right]
 \|v\|_{L^{p_*}(\bdH)}^{\,p-p_*}
 \int_{\bdH}v^{p_*-2}|h|^2\,dy.
 \label{eq:disturbed-gap-q-gt-two}
\end{equation}
\end{proposition}

\begin{proof}
We give the details of the contradiction argument.  First, we reduce to
the standard bubble.  If $v=aU[\lambda,\xi]$, then
\eqref{eq:disturbed-gap-v-normalization} keeps the amplitude $a$ in a
compact subset of $(0,\infty)$. Hence, it is suffice to prove the proposition for $v=U$.

\medskip
\noindent\textbf{Step 1: the case $p_*\le2$.}
Assume that \eqref{eq:disturbed-gap-q-le-two} is false.  There are
$0\not\equiv h_k\perp T_U\MT$ such that
$\|\nabla h_k\|_{L^p(\Hhalf)}\to0$ and
\[
 \mathcal H_U[h_k]
 <\left[(p_*-1)\ST^p+\lambda_{\mathrm T}\right]
 \|U\|_{L^{p_*}(\bdH)}^{\,p-p_*}
 \int_{\bdH}
 \frac{(U+C_0|h_k|)^{p_*}}{U^2+|h_k|^2}|h_k|^2\,dy.
\]
Put
\begin{equation}
 \varepsilon_k^2
 :=\int_{\Hhalf}
 (|\nabla U|+|\nabla h_k|)^{p-2}|\nabla h_k|^2\,dx,
 \qquad
 \phi_k:=\frac{h_k}{\varepsilon_k}.
 \label{eq:critical-normalization-scale}
\end{equation}
By \eqref{eq:M-monotonicity-comparison},
\begin{equation}
 c_p\varepsilon_k^2
 \le\mathcal H_U[h_k]
 \le C_p\varepsilon_k^2.
 \label{eq:H-epsilon-comparable}
\end{equation}
Since $p_*\le2$ implies $p<2$,
\[
 \varepsilon_k^2
 \le\int_{\Hhalf}|\nabla h_k|^p\,dx\longrightarrow0.
\]
The compactness lemma
\cite[Lemma~4.3]{TraceGradientStability} is therefore applicable.
After passing to a subsequence, it gives
\[
 \phi_k\rightharpoonup\phi\quad\text{in }\Dp(\Hhalf),
 \qquad
 \phi_k\to\phi\quad\text{a.e.\ in }\Hhalf\text{ and on }\bdH,
 \qquad
 \phi\perp T_U\MT,
\]
where
$\phi\in\Dp(\Hhalf)\cap L^2(\bdH,U^{p_*-2}dy)$.  The last assertion
of \cite[Lemma~4.3]{TraceGradientStability} is precisely
\begin{equation}
 \frac1{\varepsilon_k^2}
 \int_{\bdH}
 \frac{(U+C_0|h_k|)^{p_*}}{U^2+|h_k|^2}|h_k|^2\,dy
 \longrightarrow
 \int_{\bdH}U^{p_*-2}|\phi|^2\,dy.
 \label{eq:critical-nonlinear-boundary-limit}
\end{equation}
Thus \eqref{eq:critical-nonlinear-boundary-limit} is exactly the
nonlinear boundary convergence supplied by that lemma.

We next prove the lower limit
\begin{equation}
 \liminf_{k\to\infty}
 \frac{\mathcal H_U[h_k]}{\varepsilon_k^2}
 \ge
 \int_{\Hhalf}\mathcal A_U\nabla\phi\cdot\nabla\phi\,dx.
 \label{eq:critical-bulk-liminf}
\end{equation}
Write
\[
 g_k(x,Y)
 :=\frac1{\varepsilon_k^2}
 \mathscr M_p(\nabla U(x),\varepsilon_kY).
\]
Then $g_k\ge0$ and
\[
 \frac{\mathcal H_U[h_k]}{\varepsilon_k^2}
 =\int_{\Hhalf}g_k(x,\nabla\phi_k)\,dx.
\]
Fix $R>0$, put $B_R^+:=B_R\cap\Hhalf$, and note that
\[
 0<c_R\le|\nabla U|\le C_R
 \qquad\text{on }\overline{B_R^+}.
\]
Define
\[
 E_{k,R}
 :=\{x\in B_R^+:
 \varepsilon_k^{1/4}|\nabla\phi_k(x)|\le|\nabla U(x)|\}.
\]
On $E_{k,R}$,
$\varepsilon_k|\nabla\phi_k|/|\nabla U|\le\varepsilon_k^{3/4}$.
To explain the resulting uniform expansion, set
\[
 F(X):=|X|^{p-2}X.
\]
For $X\ne0$,
\[
 DF(X)
 =|X|^{p-2}I+(p-2)|X|^{p-4}X\otimes X.
\]
The fundamental theorem of calculus gives the exact formula
\begin{align*}
 g_k(x,Y)
 &=\frac{F(\nabla U+\varepsilon_kY)-F(\nabla U)}{\varepsilon_k}
   \cdot Y\\
 &=\int_0^1
   DF(\nabla U+t\varepsilon_kY)Y\cdot Y\,dt.
\end{align*}
On $E_{k,R}$ the whole segment
$\nabla U+t\varepsilon_k\nabla\phi_k$, $0\le t\le1$, stays within
relative distance $\varepsilon_k^{3/4}$ of $\nabla U$.  Since
$\nabla U$ is smooth and bounded away from zero on
$\overline{B_R^+}$, Taylor's formula gives, for every $R,M<\infty$,
\[
 \sup_{\substack{x\in\overline{B_R^+}\\
       \varepsilon_k^{1/4}|Y|\le M|\nabla U(x)|}}
 \left|g_k(x,Y)-\mathcal A_U(x)Y\cdot Y\right|
 \longrightarrow0.
\]

Furthermore, since $\{\phi_k\}$ is uniformly
bounded in $\Dp$, we get
\[
 |B_R^+\setminus E_{k,R}|
 \le C_R\varepsilon_k^{p/4}
 \int_{B_R^+}|\nabla\phi_k|^p\,dx
 \longrightarrow0.
\]
It follows that
\[
 \mathbf1_{E_{k,R}}\nabla\phi_k
 \rightharpoonup\nabla\phi
 \qquad\text{weakly in }L^p(B_R^+).
\]
Since $g_k\ge0$, the lower semicontinuity theorem
\cite[Theorem~3]{Ioffe1977} and the uniform expansion above give
\begin{align*}
 \liminf_{k\to\infty}
 \frac{\mathcal H_U[h_k]}{\varepsilon_k^2}
 &\ge
 \liminf_{k\to\infty}
 \int_{E_{k,R}}g_k(x,\nabla\phi_k)\,dx\\
 &=\liminf_{k\to\infty}\int_{E_{k,R}}
   \mathcal A_U(x)\nabla\phi_k\cdot \nabla\phi_k\,dx\\
 &\ge
 \int_{B_R^+}\mathcal A_U\nabla\phi\cdot\nabla\phi\,dx.
\end{align*}
Letting $R\to\infty$ proves
\eqref{eq:critical-bulk-liminf}.  In particular, the right-hand side is
finite.  Hence
\[
 \phi\in\dot W^{1,2}
 (\Hhalf;|\nabla U|^{p-2}).
\]

Divide the failed inequality by $\varepsilon_k^2$ and combine
\eqref{eq:critical-nonlinear-boundary-limit} with
\eqref{eq:critical-bulk-liminf}.  We obtain
\[
 \int_{\Hhalf}\mathcal A_U\nabla\phi\cdot\nabla\phi\,dx
 \le
 \left[(p_*-1)\ST^p+\lambda_{\mathrm T}\right]
 \|U\|_{L^{p_*}(\bdH)}^{\,p-p_*}
 \int_{\bdH}U^{p_*-2}|\phi|^2\,dy.
\]
The limit is nonzero.  Otherwise,
\eqref{eq:critical-nonlinear-boundary-limit} and the failed inequality
would imply
\[
 \limsup_{k\to\infty}
 \frac{\mathcal H_U[h_k]}{\varepsilon_k^2}\le0,
\]
contradicting the lower bound in
\eqref{eq:H-epsilon-comparable}.  The last displayed inequality now
contradicts Proposition~\ref{prop:linear-trace-gap-cp}, whose
coefficient exceeds the last one by
$\lambda_{\mathrm T}\|U\|_{L^{p_*}(\bdH)}^{p-p_*}>0$.

\medskip
\noindent\textbf{Step 2: the case $p_*>2$.}
Assume that \eqref{eq:disturbed-gap-q-gt-two} is false.  Then there are
$0\not\equiv h_k\perp T_U\MT$, with $\|\nabla h_k\|_p\to0$, such that
\[
 \mathcal H_U[h_k]
 <\left[(p_*-1)\ST^p+\lambda_{\mathrm T}\right]
 \|U\|_{L^{p_*}(\bdH)}^{\,p-p_*}
 \int_{\bdH}U^{p_*-2}|h_k|^2\,dy.
\]
Define $\varepsilon_k$ and $\phi_k$ by
\eqref{eq:critical-normalization-scale}; the comparison
\eqref{eq:H-epsilon-comparable} remains valid.

Suppose first that $1<p<2$.
The compactness lemma \cite[Lemma~4.3]{TraceGradientStability} yields,
after a subsequence,
\[
 \phi_k\rightharpoonup\phi\quad\text{in }\Dp(\Hhalf),
 \qquad
 \phi_k\to\phi
 \quad\text{strongly in }L^2(\bdH,U^{p_*-2}dy),
 \qquad
 \phi\perp T_U\MT.
\]
The failed inequality and the strong boundary convergence bounds
$\mathcal H_U[h_k]/\varepsilon_k^2$.  The proof of
\eqref{eq:critical-bulk-liminf} in Step~1 therefore applies without
change.

Suppose instead that $p\ge2$.  First,
\begin{align*}
 \varepsilon_k^2
 &\le C\int_{\Hhalf}
 \bigl(|\nabla U|^{p-2}|\nabla h_k|^2+|\nabla h_k|^p\bigr)\,dx\\
 &\le C\|\nabla U\|_p^{p-2}\|\nabla h_k\|_p^2
 +C\|\nabla h_k\|_p^p
 \longrightarrow0.
\end{align*}
Moreover, by the definition of $\varepsilon_k$,
\[
 \int_{\Hhalf}|\nabla U|^{p-2}|\nabla\phi_k|^2\,dx
 \le1.
\]
Thus, after a subsequence,
\[
 \phi_k\rightharpoonup\phi
 \quad\text{in }
 \dot W^{1,2}(\Hhalf;|\nabla U|^{p-2}).
\]
The compact weighted trace embedding
\cite[Theorem~2.4]{TraceGradientStability} gives
\[
 \phi_k\to\phi
 \quad\text{strongly in }L^2(\bdH,U^{p_*-2}dy).
\]
The orthogonality conditions pass to the limit, so
$\phi\perp T_U\MT$.

For completeness, the lower limit \eqref{eq:critical-bulk-liminf} also holds in this range.
On each $B_R^+$ the sequence $(\nabla\phi_k)$ is bounded in $L^2$.
With $E_{k,R}$ as in Step~1,
\[
 |B_R^+\setminus E_{k,R}|
 \le C_R\varepsilon_k^{1/2}
 \int_{B_R^+}|\nabla\phi_k|^2\,dx
 \longrightarrow0,
\]
and
$\mathbf1_{E_{k,R}}\nabla\phi_k\rightharpoonup\nabla\phi$ in
$L^2(B_R^+)$.  The same Taylor estimate for $g_k$, followed by weak
lower semicontinuity and then $R\to\infty$, proves
\eqref{eq:critical-bulk-liminf}.

We have therefore shown in both ranges $1<p<2$ and $p\ge2$ that
\[
 \phi_k\to\phi
 \quad\text{strongly in }L^2(\bdH,U^{p_*-2}dy),
 \qquad
 \liminf_{k\to\infty}
 \frac{\mathcal H_U[h_k]}{\varepsilon_k^2}
 \ge
 \int_{\Hhalf}\mathcal A_U\nabla\phi\cdot\nabla\phi\,dx.
\]
Dividing the failed inequality by $\varepsilon_k^2$ yields
\[
 \int_{\Hhalf}\mathcal A_U\nabla\phi\cdot\nabla\phi\,dx
 \le
 \left[(p_*-1)\ST^p+\lambda_{\mathrm T}\right]
 \|U\|_{L^{p_*}(\bdH)}^{\,p-p_*}
 \int_{\bdH}U^{p_*-2}|\phi|^2\,dy.
\]
Again $\phi\ne0$: if $\phi=0$, strong weighted boundary convergence
would make the normalized right-hand side tend to zero, whereas
\eqref{eq:H-epsilon-comparable} keeps the normalized left-hand side
bounded below by $c_p>0$.  We have reached the same contradiction with
Proposition~\ref{prop:linear-trace-gap-cp}.

This proves both estimates at $U$.  The initial transport argument and
the compact amplitude range supplied by
\eqref{eq:disturbed-gap-v-normalization} give the asserted uniform
constants for general $v\in\MT^+$.
\end{proof}

\section{Proof of Theorem~\ref{thm:trace-critical-stability}}

We use the following normalized-manifold version of
\cite[Proposition~5.3]{TraceGradientStability}, obtained by restricting
the modulation parameters to dilations and tangential translations.

\begin{lemma}[Orthogonal modulation]
\label{lem:critical-modulation}
There is a modulus $\omega(t)\to0$ as $t\to0$ such that, if
$\|\nabla(u-W)\|_p\le t$ for some $W\in\NT$, then there is
$v\in\NT$ for which, with $h=u-v$,
\begin{equation}
 h\perp T_v\NT,
 \qquad
 \|\nabla h\|_p\le\omega(t).
 \label{eq:modulation-output-cp}
\end{equation}
\end{lemma}

\begin{proof}[Proof of Theorem~\ref{thm:trace-critical-stability}]
Choose $W\in\NT$ such that
\[
 t:=\|\nabla(u-W)\|_p\le2\delta,
\]
and apply Lemma~\ref{lem:critical-modulation}.  Denote the resulting
normalized bubble by $v_0\in\NT$ and put $h_0:=u-v_0$.  Thus
\[
 h_0\perp T_{v_0}\NT,
 \qquad
 \|\nabla h_0\|_p\le\omega(t).
\]
Define
\begin{equation}
 a:=\frac{\displaystyle\int_{\bdH}v_0^{p_*-1}u\,dy}
          {\displaystyle\int_{\bdH}v_0^{p_*}\,dy},
 \qquad
 \widetilde v:=av_0,
 \qquad
 h:=u-\widetilde v.
 \label{eq:amplitude-correction-cp}
\end{equation}
By the trace inequality and the fact that the $L^{p_*}$ norm is
constant on $\NT$,
\[
 |a-1|
 \le\frac{\|h_0\|_{L^{p_*}(\bdH)}}
          {\|v_0\|_{L^{p_*}(\bdH)}}
 \le C\|\nabla h_0\|_p
 \le C\omega(t).
\]
In particular, $a>0$ for small $t$.  Moreover,
\[
 T_{\widetilde v}\MT
 =\operatorname{span}\{\widetilde v\}\oplus aT_{v_0}\NT.
\]
The definition of $a$ gives
\[
 \int_{\bdH}\widetilde v^{p_*-2}h\widetilde v\,dy=0.
\]
If $Z_0\in T_{v_0}\NT$, then
$\int_{\bdH}v_0^{p_*-1}Z_0\,dy=0$, because the $L^{p_*}$ norm is
constant on $\NT$.  Hence
\begin{align*}
 \int_{\bdH}\widetilde v^{p_*-2}h(aZ_0)\,dy
 &=a^{p_*-1}\left[
   \int_{\bdH}v_0^{p_*-2}h_0Z_0\,dy
   +(1-a)\int_{\bdH}v_0^{p_*-1}Z_0\,dy\right]\\
 &=0.
\end{align*}
Therefore $h\perp T_{\widetilde v}\MT$.  Also,
\[
 \|\nabla h\|_p
 \le\|\nabla h_0\|_p+|a-1|\|\nabla v_0\|_p
 \le C\omega(t),
 \qquad
 \Lambda_{\widetilde v}=a^{p-p_*}\longrightarrow1
 \quad\text{as }t\to0.
\]
Set
\[
 \varepsilon:=\|\nabla h\|_{L^p(\Hhalf)},
 \qquad
 R:=\|\mathcal P_{\mathrm T}(u)\|_{W^{-1,p'}(\Hhalf)}.
\]
Assume first that $\varepsilon>0$.  After decreasing the closeness
threshold, $\Lambda_{\widetilde v}$ is as close to $1$ as needed below.
Since
\[
 \|\widetilde v\|_{L^{p_*}(\bdH)}
 =\ST^{\frac{n-p}{p-1}}
 \Lambda_{\widetilde v}^{-1/(p_*-p)},
\]
condition \eqref{eq:disturbed-gap-v-normalization} also holds, and
Proposition~\ref{prop:disturbed-trace-gap-cp} applies.

Because $\widetilde v\in T_{\widetilde v}\MT$, the orthogonality gives
\begin{equation}
 \int_{\bdH}\widetilde v^{p_*-1}h\,dy=0.
 \label{eq:amplitude-orthogonality-cp}
\end{equation}
Equation \eqref{eq:EL-v-cp} then also yields
\begin{equation}
 \int_{\Hhalf}|\nabla\widetilde v|^{p-2}
 \nabla\widetilde v\cdot\nabla h\,dx=0.
 \label{eq:bulk-linear-cancellation-cp}
\end{equation}
Consequently, testing the residual by $h$ and using
\eqref{eq:M-critical-definition}--\eqref{eq:N-critical-definition},
we obtain the exact identity
\begin{equation}
 \langle\mathcal P_{\mathrm T}(u),h\rangle
 =\mathcal H_{\widetilde v}[h]
 -\int_{\bdH}\mathscr N_{p_*}(\widetilde v,h)\,dy.
 \label{eq:critical-master-identity}
\end{equation}
Hence
\begin{equation}
 R\varepsilon
 \ge
 \mathcal H_{\widetilde v}[h]
 -\int_{\bdH}\mathscr N_{p_*}(\widetilde v,h)\,dy.
 \label{eq:critical-master-bound}
\end{equation}

\medskip
\noindent\textbf{Case 1: $p_*\le2$.}
Choose $\kappa>0$ sufficiently small, depending only on $n$ and $p$,
and let $C_0=C_0(p_*,\kappa)$ be the constant from
\eqref{eq:N-q-le-two-critical}.  Then decrease the closeness threshold
so that Proposition~\ref{prop:disturbed-trace-gap-cp} applies with this
$C_0$ and $\Lambda_{\widetilde v}$ is sufficiently close to $1$.  By
\eqref{eq:N-q-le-two-critical},
\[
 \int_{\bdH}\mathscr N_{p_*}(\widetilde v,h)\,dy
 \le(p_*-1+\kappa)
 \int_{\bdH}
 \frac{(\widetilde v+C_0|h|)^{p_*}}
 {\widetilde v^2+|h|^2}|h|^2\,dy.
\]
Use \eqref{eq:disturbed-gap-q-le-two}.  Since
$\Lambda_{\widetilde v}\to1$ and
$\|\widetilde v\|_{L^{p_*}(\bdH)}^{p-p_*}
=\Lambda_{\widetilde v}/\ST^p$, the extra
$\lambda_{\mathrm T}$ term in the disturbed gap stays uniformly
positive.  Hence the preceding choices of $\kappa$ and the closeness
threshold ensure that
\begin{equation}
 R\varepsilon\ge c\mathcal H_{\widetilde v}[h].
 \label{eq:case-one-H-control}
\end{equation}
Here $p_*\le2$ implies $p<2$.  Write
$h=\varepsilon\phi$ with $\|\nabla\phi\|_p=1$.  From
\eqref{eq:M-min-form-critical},
\begin{align*}
 \mathcal H_{\widetilde v}[h]
 &\ge c\int_{\Hhalf}
 \min\{\varepsilon^p|\nabla\phi|^p,
 \varepsilon^2|\nabla\widetilde v|^{p-2}|\nabla\phi|^2\}\,dx.
\end{align*}
The last integral is at least $c\varepsilon^2$.  Indeed, on
$E=\{\varepsilon|\nabla\phi|\ge|\nabla\widetilde v|\}$ it contains the
first
term.  If $\int_E|\nabla\phi|^p\ge1/2$, this gives
$c\varepsilon^p\ge c\varepsilon^2$.  Otherwise the complement carries
at least half of the $L^p$ mass, and H\"older's inequality, together
with $\|\nabla\widetilde v\|_p\simeq_{n,p}1$, gives
\[
 \int_{E^c}|\nabla\widetilde v|^{p-2}|\nabla\phi|^2\,dx\ge c.
\]
Thus \eqref{eq:case-one-H-control} gives
$R\varepsilon\ge c\varepsilon^2$, and therefore
\begin{equation}
 \varepsilon\le CR.
 \label{eq:case-one-final-cp}
\end{equation}

\medskip
\noindent\textbf{Case 2: $p_*>2$ and $1<p<2$.}
Choose $\kappa>0$ sufficiently small and then decrease the closeness
threshold.  By \eqref{eq:N-q-gt-two-critical} and
\eqref{eq:disturbed-gap-q-gt-two},
\begin{equation}
 R\varepsilon
 \ge c\mathcal H_{\widetilde v}[h]
 -C\|h\|_{L^{p_*}(\bdH)}^{p_*}.
 \label{eq:case-two-preliminary-cp}
\end{equation}
The preceding mixed-remainder argument gives
$\mathcal H_{\widetilde v}[h]\ge c\varepsilon^2$, while the trace
inequality gives
$\|h\|_{p_*}^{p_*}\le C\varepsilon^{p_*}$.  Since $p_*>2$, the last
term in \eqref{eq:case-two-preliminary-cp} is absorbed for small
$\varepsilon$.  Hence again
\begin{equation}
 \varepsilon\le CR.
 \label{eq:case-two-final-cp}
\end{equation}

\medskip
\noindent\textbf{Case 3: $p\ge2$.}
Here $p_*>p$.  The same boundary argument gives
\begin{equation}
 R\varepsilon
 \ge c\mathcal H_{\widetilde v}[h]
 -C\|h\|_{L^{p_*}(\bdH)}^{p_*}.
 \label{eq:case-three-preliminary-cp}
\end{equation}
By \eqref{eq:M-p-ge-two-critical} and the trace inequality,
\[
 \mathcal H_{\widetilde v}[h]\ge c\varepsilon^p,
 \qquad
 \|h\|_{L^{p_*}(\bdH)}^{p_*}\le C\varepsilon^{p_*}.
\]
Since $p_*>p$, the boundary error is absorbed for small
$\varepsilon$.  Thus
\begin{equation}
 \varepsilon^{p-1}\le CR.
 \label{eq:case-three-final-cp}
\end{equation}

We now recover the normalized amplitude, allowing also
$\varepsilon=0$.  For $X,Y\in\R^n$, $b>0$, $c\in\R$, and $q>1$, define
the energy remainders from
\cite[Subsection~4.1]{TraceGradientStability} by
\begin{align*}
 \mathscr G_p(X,Y)
 &:={}|X+Y|^p-|X|^p-p|X|^{p-2}X\cdot Y,\\
 \mathscr B_q(b,c)
 &:={}|b+c|^q-b^q-qb^{q-1}c,\\
 \mathcal G_{\widetilde v}[h]
 &:={\frac{2}{p}}\int_{\Hhalf}
 \mathscr G_p(\nabla\widetilde v,\nabla h)\,dx.
\end{align*}

\begin{claim}
If $\delta$ is sufficiently small, then $|a-1|\le CR$.
\end{claim}

Indeed, testing the residual by $u$ and using
\eqref{eq:amplitude-orthogonality-cp}--
\eqref{eq:bulk-linear-cancellation-cp}, we obtain the exact identity
\begin{align}
 \langle\mathcal P_{\mathrm T}(u),u\rangle
 &=\int_{\Hhalf}|\nabla u|^p\,dx
   -\int_{\bdH}|u|^{p_*}\,dy\notag\\
 &=a^p(1-a^{p_*-p})
   \ST^{\frac{p(n-1)}{p-1}}
   +\frac p2\mathcal G_{\widetilde v}[h]
   -\int_{\bdH}\mathscr B_{p_*}(\widetilde v,h)\,dy.
 \label{eq:amplitude-energy-identity-cp}
\end{align}
The boundary Taylor estimates in
\cite[Subsection~4.1]{TraceGradientStability}, together with the
disturbed energy gap
\cite[Proposition~4.5 and Corollary~4.6]{TraceGradientStability}, give
\begin{equation}
 \int_{\bdH}\mathscr B_{p_*}(\widetilde v,h)\,dy
 \le C\mathcal G_{\widetilde v}[h]
      +C\|h\|_{L^{p_*}(\bdH)}^{p_*}.
 \label{eq:boundary-energy-remainder-control-cp}
\end{equation}

Since $a\to1$, the first term on the right-hand side of
\eqref{eq:amplitude-energy-identity-cp} has absolute value bounded
below by $c|a-1|$.  Moreover,
$|\langle\mathcal P_{\mathrm T}(u),u\rangle|
\le R\|\nabla u\|_p\le CR$.  Hence
\begin{equation}
 |a-1|\le CR+C\mathcal G_{\widetilde v}[h]+C\varepsilon^{p_*}.
 \label{eq:amplitude-before-monotonicity-cp}
\end{equation}
Lemma~\ref{lem:pointwise-critical-remainders} and
\cite[Lemma~4.1]{TraceGradientStability} show that
$\mathcal G_{\widetilde v}[h]\le C\mathcal H_{\widetilde v}[h]$.
The estimates in the three cases above (or simply
$\mathcal H_{\widetilde v}[h]=0$ if $\varepsilon=0$)
therefore yield
\[
 |a-1|
 \le CR+C\mathcal H_{\widetilde v}[h]+C\varepsilon^{p_*}
 \le CR+C\varepsilon R+C\varepsilon^{p_*}
 \le CR.
\]
This proves the claim.

Finally, let
$\alpha:=\min\{1,1/(p-1)\}$.  The three cases, together with the
trivial estimate when $\varepsilon=0$, give
$\varepsilon\le CR^\alpha$.  After decreasing $\delta$ once more, the
continuity of the residual at $\NT$ ensures that $R\le1$.  Therefore
\begin{equation}
 \|\nabla h_0\|_{L^p(\Hhalf)}
 \le\varepsilon+C|a-1|
 \le C(R^\alpha+R)
 \le CR^\alpha.
 \label{eq:normalized-remainder-final-cp}
\end{equation}
Since
$\alpha=1/\max\{1,p-1\}$, taking $v_0\in\NT$ as the normalized
bubble in the statement and raising
\eqref{eq:normalized-remainder-final-cp} to the power
$\max\{1,p-1\}$ proves
\eqref{eq:critical-point-stability-main}.  The distance estimate
\eqref{eq:critical-point-distance-form} follows immediately.

\end{proof}

\section{Global compactness and the energy-window consequences}
\label{sec:global-compactness}

In this section we prove Theorem~\ref{thm:global-compactness}, then
deduce Corollary~\ref{thm:trace-nu-bubble-compactness} and
Corollary~\ref{cor:global-one-bubble-critical-stability}.  Besides the
trace exponent $p_*$, we use
\[
 p^*:=\frac{np}{n-p},
 \qquad
 p_*':=\frac{p_*}{p_*-1}.
\]
Recall that
\[
 \Dp(\Hhalf)
 :=\{u\in L^{p^*}(\Hhalf):\nabla u\in L^p(\Hhalf)\},
 \qquad
 \|u\|_{\Dp(\Hhalf)}:=\|\nabla u\|_{L^p(\Hhalf)}.
\]

During the extraction it is convenient to use the derivative
$P_+(u)\in W^{-1,p'}(\Hhalf)$ of the positive functional
\eqref{eq:positive-functional}, defined by
\begin{equation}
 \langle P_+(u),\varphi\rangle
 :=\int_{\Hhalf}|\nabla u|^{p-2}\nabla u\cdot\nabla\varphi\,dx
 -\int_{\bdH}(u_+)^{p_*-1}\varphi\,dy.
 \label{eq:positive-residual}
\end{equation}
The signed and positive residuals satisfy
\begin{equation}
 \|P_+(u)-\PT(u)\|_{W^{-1,p'}(\Hhalf)}
 \le \ST^{-1}\|u_-\|_{L^{p_*}(\bdH)}^{p_*-1}.
 \label{eq:positive-signed-comparison}
\end{equation}
We fix $W:=W[1,0]$ and recall the notation
\[
 W[\lambda,\xi]=G_{\lambda,\xi}W.
\]
The functional, both residual norms, the gradient norm, and the
critical trace norm are invariant under $G_{\lambda,\xi}$.

\subsection{Nonlinear splitting and local compactness}

We first recall the norm-splitting form of the Br\'ezis--Lieb lemma.

\begin{classicalBL}[Original scalar form]
Let $(\Omega,\mu)$ be a measure space, let $0<q<\infty$, and let
$(f_k)\subset L^q(\Omega,\mu)$ satisfy
\[
 f_k\longrightarrow f\quad\mu\text{-a.e.},
 \qquad
 \sup_k\int_\Omega|f_k|^q\,d\mu<\infty.
\]
Then $f\in L^q(\Omega,\mu)$ and
\begin{equation*}
 \lim_{k\to\infty}
 \left(
  \int_\Omega|f_k|^q\,d\mu
  -\int_\Omega|f_k-f|^q\,d\mu
 \right)
 =\int_\Omega|f|^q\,d\mu.
 \tag{BL}
\end{equation*}
The same statement holds for finite-dimensional vector-valued
functions, with $|\cdot|$ denoting the Euclidean norm.
\end{classicalBL}
This is the classical Br\'ezis--Lieb lemma~\cite{BrezisLieb1983}.

We next prove the derivative splittings needed for the
Euler--Lagrange operator.  They are the counterparts
of Lemma~3.2 of Mercuri--Willem~\cite{MercuriWillem2010}.

\begin{lemma}[Nonlinear Br\'ezis--Lieb splitting]
\label{lem:nonlinear-BL}
Let $q>1$.
\begin{enumerate}
\item If $X_k\to X$ almost everywhere, with $(X_k)$ bounded in
$L^q(\Omega;\R^N)$, then
\[
 |X_k|^{q-2}X_k
 -|X_k-X|^{q-2}(X_k-X)-|X|^{q-2}X
 \longrightarrow0
 \quad\text{in }L^{q'}(\Omega;\R^N).
\]
\item If $s_k\to s$ almost everywhere, with $(s_k)$ bounded in
$L^q(\Omega)$, then
\begin{align*}
 &(s_k)_+^{q-1}-(s_k-s)_+^{q-1}-s_+^{q-1}
 \longrightarrow0\quad\text{in }L^{q'}(\Omega),\\
 &\int_\Omega (s_k)_+^q
 -\int_\Omega (s_k-s)_+^q
 -\int_\Omega s_+^q\longrightarrow0.
\end{align*}
\end{enumerate}
\end{lemma}

\begin{proof}
Set
\[
 A_q(X):=|X|^{q-2}X,
 \qquad
 B_q(s):=(s_+)^{q-1},
 \qquad
 F_q(s):=(s_+)^q.
\]
We shall use three elementary estimates.  If $1<q<2$, then
\begin{equation}
 |A_q(a+b)-A_q(a)|+|B_q(\alpha+\beta)-B_q(\alpha)|
 \le C_q\bigl(|b|^{q-1}+|\beta|^{q-1}\bigr).
 \label{eq:holder-nonlinearity}
\end{equation}
To verify this, first suppose $|a|\le2|b|$.  Then $|a+b|\le3|b|$,
and homogeneity gives
\[
 |A_q(a+b)-A_q(a)|
 \le |a+b|^{q-1}+|a|^{q-1}\le C_q|b|^{q-1}.
\]
If $|a|>2|b|$, then $|a+tb|\ge|a|/2$ for $0\le t\le1$.  Since
$\|DA_q(\zeta)\|\le C_q|\zeta|^{q-2}$ for $\zeta\ne0$, integration
along the segment gives
\[
 |A_q(a+b)-A_q(a)|
 \le C_q|a|^{q-2}|b|\le C_q|b|^{q-1}.
\]
The scalar estimate follows from
$|(x_+)^{q-1}-(y_+)^{q-1}|\le|x-y|^{q-1}$.  This proves
\eqref{eq:holder-nonlinearity}.  If $q\ge2$, the mean-value theorem
gives
\begin{align}
 |A_q(a+b)-A_q(a)|
 &\le C_q\bigl(|a|^{q-2}|b|+|b|^{q-1}\bigr),
 \label{eq:lipschitz-vector-nonlinearity}\\
 |B_q(\alpha+\beta)-B_q(\alpha)|
 &\le C_q\bigl(|\alpha|^{q-2}|\beta|+|\beta|^{q-1}\bigr).
 \label{eq:lipschitz-positive-nonlinearity}
\end{align}
Raising these estimates to $q'=q/(q-1)$ and applying Young's
inequality shows that, for either $\Phi=A_q$ or $\Phi=B_q$ and every
$\varepsilon>0$,
\begin{equation}
 |\Phi(a+b)-\Phi(a)|^{q'}
 \le \varepsilon|a|^q+C_{q,\varepsilon}|b|^q.
 \label{eq:derivative-epsilon-estimate}
\end{equation}
For $A_q$ the variables are vectors, while for $B_q$ they are scalars.
In the range $1<q<2$, \eqref{eq:derivative-epsilon-estimate} follows
directly from \eqref{eq:holder-nonlinearity}.  Finally, the mean-value
theorem and Young's inequality also give
\begin{equation}
 |F_q(a+b)-F_q(a)|
 \le \varepsilon|a|^q+C_{q,\varepsilon}|b|^q.
 \label{eq:positive-energy-epsilon-estimate}
\end{equation}

We prove part~1.  Write $Y_k:=X_k-X$ and
\[
 R_k:=A_q(Y_k+X)-A_q(Y_k)-A_q(X).
\]
Then $Y_k\to0$ and $R_k\to0$ almost everywhere.  From
\eqref{eq:derivative-epsilon-estimate}, after changing the constants,
\[
 |R_k|^{q'}
 \le \varepsilon|Y_k|^q+C_{q,\varepsilon}|X|^q.
\]
The sequence $(Y_k)$ is bounded in $L^q$.  Moreover,
\[
 \bigl(|R_k|^{q'}-\varepsilon|Y_k|^q\bigr)_+
 \longrightarrow0
\]
almost everywhere and is dominated by $C_{q,\varepsilon}|X|^q$.
Dominated convergence therefore yields
\[
 \limsup_{k\to\infty}\|R_k\|_{q'}^{q'}
 \le \varepsilon\sup_k\|Y_k\|_q^q.
\]
Letting $\varepsilon\downarrow0$ proves part~1.

For the first assertion in part~2, take $y_k:=s_k-s$ and repeat the
same argument with $B_q$ in place of $A_q$.  For the last assertion put
\[
 \rho_k:=F_q(y_k+s)-F_q(y_k)-F_q(s).
\]
Then $\rho_k\to0$ almost everywhere, and
\eqref{eq:positive-energy-epsilon-estimate} gives
\[
 |\rho_k|\le\varepsilon|y_k|^q+C_{q,\varepsilon}|s|^q.
\]
The positive-part domination argument used above, now in $L^1$, gives
$\|\rho_k\|_{L^1}\to0$.  This is precisely the claimed integral
splitting.
\end{proof}

\begin{lemma}[Splitting Lemma]
\label{lem:local-compactness}
Let $(z_k)\subset\Dp(\Hhalf)$ be bounded,
\[
 z_k\rightharpoonup z\quad\text{in }\Dp(\Hhalf),
 \qquad
 \|P_+(z_k)\|_{W^{-1,p'}(\Hhalf)}\longrightarrow0.
\]
Then, after passing to a subsequence,
\begin{equation}
 z_k\to z\quad\text{a.e.\ on }\bdH,
 \qquad
 \nabla z_k\to\nabla z\quad\text{a.e.\ in }\Hhalf,
 \label{eq:ae-convergence}
\end{equation}
 $P_+(z)=0$ and $z\ge 0$.  If $r_k:=z_k-z$, then
\begin{align}
 \|\nabla z_k\|_p^p
 &=\|\nabla z\|_p^p+\|\nabla r_k\|_p^p+o(1),
 \label{eq:bulk-splitting}\\
 \int_{\bdH}(z_k)_+^{p_*}\,dy
 &=\int_{\bdH}z_+^{p_*}\,dy
 +\int_{\bdH}(r_k)_+^{p_*}\,dy+o(1),
 \label{eq:boundary-splitting}\\
 J_+(z_k)&=J_+(z)+J_+(r_k)+o(1),
 \label{eq:functional-splitting}\\
 \|P_+(r_k)\|_{W^{-1,p'}(\Hhalf)}&\longrightarrow0.
 \label{eq:residual-splitting}
\end{align}
In addition,
\begin{equation}
 \|\nabla(r_k)_-\|_{L^p(\Hhalf)}\longrightarrow0,
 \qquad
 \|(r_k)_-\|_{L^{p_*}(\bdH)}\longrightarrow0.
 \label{eq:negative-part-small}
\end{equation}
\end{lemma}

\begin{proof}
Let $K\subset\overline{\Hhalf}$ be compact.  Since $p<p^*$,
boundedness in $\Dp(\Hhalf)$ implies boundedness in
$W^{1,p}(U\cap\Hhalf)$ for every bounded smooth neighborhood $U$ of
$K$.  The Rellich theorem and the compact local trace embedding give,
after a diagonal extraction,
\begin{align}
 z_k&\longrightarrow z
 &&\text{strongly in }L^q_{\mathrm{loc}}(\Hhalf),
 &&1\le q<p^*,
 \label{eq:local-bulk-compactness}\\
 z_k&\longrightarrow z
 &&\text{strongly in }L^r_{\mathrm{loc}}(\bdH),
 &&1\le r<p_*.
 \label{eq:local-trace-compactness}
\end{align}
In particular, $z_k\to z$ almost everywhere both in $\Hhalf$ and on
$\bdH$.

Put
\[
 w_k:=z_k-z,
 \qquad
 A(X):=|X|^{p-2}X,
 \qquad
 T(s):=\max\{-1,\min\{s,1\}\}.
\]
Fix a compact set $K\subset\overline{\Hhalf}$ and choose
$\eta\in C_c^\infty(\overline{\Hhalf})$, $0\le\eta\le1$, with
$\eta=1$ on $K$.  The functions $\eta T(w_k)$ are bounded in
$\Dp(\Hhalf)$.  Also, $T(w_k)\to0$ strongly in every finite local
Lebesgue space by dominated convergence, while
\[
 \nabla T(w_k)=\mathbf1_{\{|w_k|<1\}}\nabla w_k
\]
is bounded in local $L^p$.  Hence
$\nabla T(w_k)\rightharpoonup0$ locally in $L^p$.

Using $\eta T(w_k)$ as a test function in $P_+(z_k)$ and expanding the
gradient of the product gives
\begin{align}
 I_k
 &:=\int_{\Hhalf}\eta
 \bigl(A(\nabla z_k)-A(\nabla z)\bigr)
 \cdot\nabla T(w_k)\,dx                                      \notag\\
 &=\langle P_+(z_k),\eta T(w_k)\rangle
   +\int_{\bdH}(z_k)_+^{p_*-1}\eta T(w_k)\,dy                \notag\\
 &\quad-\int_{\Hhalf}T(w_k)A(\nabla z_k)\cdot\nabla\eta\,dx
   -\int_{\Hhalf}\eta A(\nabla z)\cdot\nabla T(w_k)\,dx.
 \label{eq:expanded-truncation-identity}
\end{align}
Every term on the last two lines tends to zero.  The first does so
because $P_+(z_k)\to0$ and the test functions are bounded in
$\Dp(\Hhalf)$.  For the boundary term, $(z_k)_+^{p_*-1}$ is bounded
in $L^{p_*'}(\bdH)$, whereas $\eta T(w_k)\to0$ strongly in
$L^{p_*}(\bdH)$ because it is bounded, compactly supported, and
converges almost everywhere.  The cutoff term tends to zero by
H\"older's inequality, the local $L^{p'}$ boundedness of
$A(\nabla z_k)$, and the strong local $L^p$ convergence of $T(w_k)$.
The last term tends to zero by the weak convergence of
$\nabla T(w_k)$.  Thus
\begin{equation}
 I_k\longrightarrow0.
 \label{eq:truncation-monotonicity}
\end{equation}

We now prove that, after passing to a subsequence,
 $\nabla z_k\to\nabla z$ a.e. in $\Hhalf$.  By strict
monotonicity of $A$,
\[
 M_k:=\bigl(A(\nabla z_k)-A(\nabla z)\bigr)
       \cdot(\nabla z_k-\nabla z)\ge0,
\]
and \eqref{eq:truncation-monotonicity} says
\begin{equation}
 \int_{K\cap\{|w_k|<1\}}M_k\,dx\longrightarrow0.
 \label{eq:local-monotonicity-defect}
\end{equation}
Fix $\varepsilon>0$ and $R>0$.  On the compact set of pairs $(a,b)$
such that $|a|+|b|\le R$ and $|a-b|\ge\varepsilon$, continuity and
strict monotonicity give
\[
 (A(a)-A(b))\cdot(a-b)\ge c(\varepsilon,R)>0.
\]
It follows from \eqref{eq:local-monotonicity-defect} that the measure
of
\[
 \left\{x\in K:|w_k|<1,\ |
 \nabla z_k-\nabla z|\ge\varepsilon,
 |\nabla z_k|+|\nabla z|\le R\right\}
\]
tends to zero.  The measure of $K\cap\{|w_k|\ge1\}$ tends to zero by
\eqref{eq:local-bulk-compactness}; moreover, Chebyshev's inequality
and the local $L^p$ bounds make
\[
 \sup_k\bigl|\{x\in K:|\nabla z_k|+|\nabla z|>R\}\bigr|
\]
arbitrarily small as $R\to\infty$.  Therefore
$\nabla z_k\to\nabla z$ in measure on $K$.  Extracting a subsequence
on an exhaustion of $\Hhalf$ proves
$\nabla z_k\to\nabla z$ almost everywhere in $\Hhalf$.

Now
\[
 A(\nabla z_k)\to A(\nabla z)\quad\text{a.e.\ in }\Hhalf,
 \qquad
 (z_k)_+^{p_*-1}\to z_+^{p_*-1}\quad\text{a.e.\ on }\bdH.
\]
The first sequence is bounded in $L^{p'}(\Hhalf)$ and the second in
$L^{p_*'}(\bdH)$.  A bounded sequence in $L^s$, $1<s<\infty$, which
converges almost everywhere, converges weakly to its pointwise limit.
Consequently, for every $\varphi\in\Dp(\Hhalf)$,
\[
 \langle P_+(z),\varphi\rangle
 =\lim_{k\to\infty}\langle P_+(z_k),\varphi\rangle=0.
\]
Thus $P_+(z)=0$.  Testing this identity with $-z_-$ also shows
$z\ge0$.

The classical Br\'ezis--Lieb lemma, in its vector form, applied to
$\nabla z_k=\nabla r_k+\nabla z$, gives
\eqref{eq:bulk-splitting}.  The last assertion of
Lemma~\ref{lem:nonlinear-BL}, with $q=p_*$, gives
\eqref{eq:boundary-splitting}; subtracting the two identities in the
definition of $J_+$ gives \eqref{eq:functional-splitting}.

The derivative parts of Lemma~\ref{lem:nonlinear-BL} give
\begin{align}
 A(\nabla z_k)-A(\nabla r_k)-A(\nabla z)&\longrightarrow0
 &&\text{in }L^{p'}(\Hhalf),
 \label{eq:bulk-derivative-splitting}\\
 (z_k)_+^{p_*-1}-(r_k)_+^{p_*-1}-z_+^{p_*-1}&\longrightarrow0
 &&\text{in }L^{p_*'}(\bdH).
 \label{eq:trace-derivative-splitting}
\end{align}
By the trace inequality, these convergences imply
\[
 \|P_+(z_k)-P_+(r_k)-P_+(z)\|_{W^{-1,p'}(\Hhalf)}\longrightarrow0.
\]
Since $P_+(z_k)\to0$ and $P_+(z)=0$, this proves
\eqref{eq:residual-splitting}.

The Lipschitz truncation $(r_k)_-$ belongs to $\Dp(\Hhalf)$.  Testing
$P_+(r_k)$ with $-(r_k)_-$ gives the exact identity
\[
 \langle P_+(r_k),-(r_k)_-\rangle
 =\int_{\Hhalf}|\nabla(r_k)_-|^p\,dx,
\]
because $(r_k)_+(r_k)_-=0$ on the boundary.  Hence
\[
 \|\nabla(r_k)_-\|_p^p
 \le \|P_+(r_k)\|_{W^{-1,p'}(\Hhalf)}\|\nabla(r_k)_-\|_p,
\]
which proves the first convergence in \eqref{eq:negative-part-small};
the second follows from \eqref{eq:sharp-trace-cp}.  
\end{proof}

\subsection{Extraction of a nonzero boundary profile}

\begin{lemma}
\label{lem:uniform-alternative}
Let $(r_k)\subset\Dp(\Hhalf)$ be bounded and suppose that
\[
 \|P_+(r_k)\|_{W^{-1,p'}(\Hhalf)}\longrightarrow0.
\]
Then either
\begin{equation}
 \|(r_k)_+\|_{L^{p_*}(\bdH)}\longrightarrow0,
 \label{eq:vanishing-alternative}
\end{equation}
or, after a subsequence,
\begin{equation}
 \liminf_{k\to\infty}
 \int_{\bdH}(r_k)_+^{p_*}\,dy\ge A_{\mathrm T}.
 \label{eq:mass-gap}
\end{equation}
\end{lemma}

\begin{proof}
Testing the residual by $r_k$ gives
\begin{equation}
 \int_{\Hhalf}|\nabla r_k|^p\,dx
 =\int_{\bdH}(r_k)_+^{p_*}\,dy+o(1).
 \label{eq:asymptotic-Nehari}
\end{equation}
Here boundedness absorbs the dual error.  Applying
\eqref{eq:sharp-trace-cp} to $(r_k)_+$ gives
\[
 \ST^p
 \left(\int_{\bdH}(r_k)_+^{p_*}\,dy\right)^{p/p_*}
 \le \int_{\Hhalf}|\nabla(r_k)_+|^p\,dx
 \le \int_{\Hhalf}|\nabla r_k|^p\,dx.
\]
Together with \eqref{eq:asymptotic-Nehari}, this says that a nonzero
limiting mass is at least
$\ST^{pp_*/(p_*-p)}=\ST^{p(n-1)/(p-1)}=A_{\mathrm T}$.
\end{proof}

\begin{lemma}[Boundary concentration produces a nonzero profile]
\label{lem:profile-extraction}
Let $(r_k)\subset\Dp(\Hhalf)$ be bounded and suppose that
\[
 \|P_+(r_k)\|_{W^{-1,p'}(\Hhalf)}\longrightarrow0
\]
and that \eqref{eq:vanishing-alternative} fails.  Then there are
$\lambda_k>0$, $\xi_k\in\R^{n-1}$ and a nontrivial nonnegative
solution $V\in\Dp(\Hhalf)$ of
\eqref{eq:normalized-trace-equation} such that
\begin{equation}
 G_{\lambda_k,\xi_k}^{-1}r_k\rightharpoonup V
 \quad\text{in }\Dp(\Hhalf).
 \label{eq:profile-weak-limit}
\end{equation}
If
\[
 \widetilde r_k:=r_k-G_{\lambda_k,\xi_k}V,
\]
then
\begin{align}
 \|\nabla r_k\|_p^p
 &=\|\nabla V\|_p^p+\|\nabla\widetilde r_k\|_p^p+o(1),
 \label{eq:profile-norm-splitting}\\
 J_+(r_k)&=J_+(V)+J_+(\widetilde r_k)+o(1),
 \label{eq:profile-energy-splitting}\\
 \|P_+(\widetilde r_k)\|_{W^{-1,p'}(\Hhalf)}&\longrightarrow0.
 \label{eq:profile-residual-splitting}
\end{align}
\end{lemma}

\begin{proof}
Choose a number $\delta>0$, depending only on $n,p$, sufficiently
small as specified below and also satisfying $2\delta<A_{\mathrm T}$.
By Lemma~\ref{lem:uniform-alternative}, the total positive boundary
mass is eventually larger than $2\delta$.  Introduce the boundary
L\'evy concentration function
\[
 Q_k(\rho)
 :=\sup_{\xi\in\R^{n-1}}
 \int_{D_\rho(\xi)}(r_k)_+^{p_*}\,dy,
\]
where $D_\rho(\xi)\subset\bdH$ is the $(n-1)$-dimensional disk.  The
standard continuity argument for $Q_k$ gives $\lambda_k>0$ such that
$Q_k(\lambda_k)=\delta$.  Choose a center at which the supremum is
attained up to an error smaller than $\delta/2$, and denote it by
$\xi_k$.  After rescaling,
\begin{equation}
 \sup_{\xi\in\R^{n-1}}
 \int_{D_1(\xi)}(z_k)_+^{p_*}\,dy=\delta,
 \qquad
 \int_{D_1(0)}(z_k)_+^{p_*}\,dy\ge\frac\delta2,
 \label{eq:normalized-concentration}
\end{equation}
where $z_k:=G_{\lambda_k,\xi_k}^{-1}r_k$.

The sequence $(z_k)$ is bounded in $\Dp(\Hhalf)$ and
$P_+(z_k)\to0$.  Passing to a subsequence, let
$z_k\rightharpoonup V$ in $\Dp(\Hhalf)$.  We claim that
$V\not\equiv0$.  Suppose otherwise.  The local bulk and trace
compactness in Lemma~\ref{lem:local-compactness} then give
\begin{equation}
 z_k\longrightarrow0\quad\text{in }L^p(B_3\cap\Hhalf)
 \quad\text{and in }L^p(D_3(0)).
 \label{eq:local-vanishing-after-normalization}
\end{equation}
This subcritical convergence alone does not contradict
\eqref{eq:normalized-concentration}; we use the equation to upgrade it
at the critical exponent.

Choose $\eta\in C_c^\infty(\overline\Hhalf)$ such that
$0\le\eta\le1$, $\eta=1$ on $B_1\cap\overline\Hhalf$, and
$\operatorname{supp}\eta\subset B_2$.  Cover $D_2(0)$ by
$N=N(n)$ unit disks.  The first part of
\eqref{eq:normalized-concentration} implies
\begin{equation}
 \int_{D_2(0)}(z_k)_+^{p_*}\,dy\le N\delta.
 \label{eq:small-mass-on-fixed-disk}
\end{equation}
The functions $\eta^p(z_k)_+$ are bounded in $\Dp(\Hhalf)$, so they
may be used as test functions in $P_+(z_k)=o(1)$.  Since
$\nabla(z_k)_+=\mathbf1_{\{z_k>0\}}\nabla z_k$, this gives
\begin{align*}
 \int_{\Hhalf}\eta^p|\nabla(z_k)_+|^p\,dx
 &=\int_{\bdH}\eta^p(z_k)_+^{p_*}\,dy\\
 &\quad-p\int_{\Hhalf}\eta^{p-1}(z_k)_+
 |\nabla(z_k)_+|^{p-2}\nabla(z_k)_+\cdot\nabla\eta\,dx+o(1).
\end{align*}
Young's inequality, followed by
\eqref{eq:local-vanishing-after-normalization}, therefore yields
\begin{equation}
 X_k:=\int_{\Hhalf}|\nabla(\eta(z_k)_+)|^p\,dx
 \le C\int_{\bdH}\eta^p(z_k)_+^{p_*}\,dy+o(1).
 \label{eq:localized-gradient-control}
\end{equation}
On the other hand, H\"older's inequality,
\eqref{eq:small-mass-on-fixed-disk}, and the trace inequality give
\begin{align}
 \int_{\bdH}\eta^p(z_k)_+^{p_*}\,dy
 &=\int_{D_2(0)}(z_k)_+^{p_*-p}
   \bigl(\eta(z_k)_+\bigr)^p\,dy\notag\\
 &\le (N\delta)^{1-p/p_*}
 \left(\int_{\bdH}\bigl(\eta(z_k)_+\bigr)^{p_*}\,dy
 \right)^{p/p_*}\notag\\
 &\le \ST^{-p}(N\delta)^{1-p/p_*}X_k.
 \label{eq:localized-boundary-absorption}
\end{align}
At the beginning choose $\delta$, still with
$2\delta<A_{\mathrm T}$, so small that
\[
 C\ST^{-p}(N\delta)^{1-p/p_*}<1.
\]
Combining \eqref{eq:localized-gradient-control} and
\eqref{eq:localized-boundary-absorption} gives $X_k\to0$.  By the
trace inequality and $\eta=1$ on $D_1(0)$,
\[
 \int_{D_1(0)}(z_k)_+^{p_*}\,dy\longrightarrow0,
\]
contrary to the second part of
\eqref{eq:normalized-concentration}.  Thus $V\not\equiv0$.

Lemma~\ref{lem:local-compactness}, applied to $(z_k)$, shows that
$P_+(V)=0$ and gives the splitting after subtracting $V$.
Transporting the conclusions back by $G_{\lambda_k,\xi_k}$ gives
\eqref{eq:profile-norm-splitting}--
\eqref{eq:profile-residual-splitting}.  Testing the equation for $V$
by $-V_-$ gives $V_-=0$.  Since $V$ is nonzero, the strong maximum
principle makes it positive in $\Hhalf$.
\end{proof}

\begin{lemma}[Energy carried by a profile]
\label{lem:profile-gap}
Every nontrivial nonnegative solution $V\in\Dp(\Hhalf)$ of
\eqref{eq:normalized-trace-equation} satisfies
\begin{equation}
 \|\nabla V\|_{L^p(\Hhalf)}^p
 =\|V\|_{L^{p_*}(\bdH)}^{p_*}\ge A_{\mathrm T},
 \qquad
 J_+(V)\ge E_{\mathrm T}.
 \label{eq:profile-gap}
\end{equation}
Equality holds after the classification theorem \cite[Theorem~1.1]{Zhou2024}, because every such
$V$ is a normalized trace bubble.
\end{lemma}

\begin{proof}
Testing the equation by $V$ gives the first equality.  Combining it
with \eqref{eq:sharp-trace-cp} yields the lower bound for the gradient
energy.  Finally,
\[
 J_+(V)
 =\left(\frac1p-\frac1{p_*}\right)\|\nabla V\|_p^p
 =\frac{p-1}{p(n-1)}\|\nabla V\|_p^p.
\]
The equality statement follows from \cite[Theorem~1.1]{Zhou2024}.
\end{proof}

\begin{lemma}
\label{lem:dislocation-facts}
The transformations in \eqref{eq:trace-dislocation} satisfy
\begin{equation}
 G_{\lambda,\xi}G_{\mu,\eta}
 =G_{\lambda\mu,\,\xi+\lambda\eta},
 \qquad
 G_{\lambda,\xi}^{-1}=G_{\lambda^{-1},-\xi/\lambda}.
 \label{eq:dislocation-group-law}
\end{equation}
Let $a_k>0$ and $b_k\in\R^{n-1}$.
\begin{enumerate}
\item If $a_k\to a\in(0,\infty)$ and $b_k\to b$, then
\[
 v_k\rightharpoonup v\ \text{in }\Dp(\Hhalf)
 \quad\Longrightarrow\quad
 G_{a_k,b_k}v_k\rightharpoonup G_{a,b}v
 \ \text{in }\Dp(\Hhalf).
\]
\item If
\begin{equation}
 a_k+\frac1{a_k}+\frac{|b_k|^2}{a_k}\longrightarrow\infty,
 \label{eq:abstract-parameter-divergence}
\end{equation}
then $G_{a_k,b_k}v\rightharpoonup0$ in $\Dp(\Hhalf)$ for every fixed
$v\in\Dp(\Hhalf)$.
\end{enumerate}
Moreover, if the left side of
\eqref{eq:abstract-parameter-divergence} is bounded, then $(a_k,b_k)$
has a subsequence converging in $(0,\infty)\times\R^{n-1}$.
\end{lemma}

\begin{proof}
The identities \eqref{eq:dislocation-group-law} follow by direct
calculation.  To prove part~1, introduce the isometry on vector fields
\[
 (T_{a,b}F)(x,t)
 :=a^{-n/p}F\left(\frac{x-b}{a},\frac{t}{a}\right),
 \qquad F\in L^p(\Hhalf;\R^n).
\]
Then $\nabla(G_{a,b}v)=T_{a,b}\nabla v$.  The adjoint action on
$L^{p'}(\Hhalf;\R^n)$ is
\[
 (T_{a,b}^*\phi)(y,s):=a^{n/p'}\phi(b+ay,as).
\]
If $a_k\to a\in(0,\infty)$ and $b_k\to b$, then
$T_{a_k,b_k}^*\phi\to T_{a,b}^*\phi$ strongly in
$L^{p'}(\Hhalf;\R^n)$, first for smooth compactly supported vector
fields and then, by density and the isometry property, for every
$\phi\in L^{p'}(\Hhalf;\R^n)$.  Hence
\begin{align*}
 \langle\nabla(G_{a_k,b_k}v_k),\phi\rangle
 &=\langle\nabla v_k,T_{a_k,b_k}^*\phi\rangle\\
 &=\langle\nabla v_k,T_{a,b}^*\phi\rangle
 +\langle\nabla v_k,
 (T_{a_k,b_k}^*-T_{a,b}^*)\phi\rangle\\
 &\longrightarrow
 \langle\nabla v,T_{a,b}^*\phi\rangle
 =\langle\nabla(G_{a,b}v),\phi\rangle,
\end{align*}
because $(v_k)$ is bounded in $\Dp(\Hhalf)$.  This proves part~1.

For part~2 it suffices, again by density, to consider a smooth compactly
supported $v$.  After passing to a subsequence,
\eqref{eq:abstract-parameter-divergence} reduces to one of the three
cases
\[
 a_k\to0,
 \qquad a_k\to\infty,
 \qquad 0<c\le a_k\le C\ \text{ and }\ |b_k|\to\infty.
\]
In the first case the support concentrates in sets whose measure tends
to zero; in the second the gradient converges uniformly to zero on
every fixed compact set; in the third the support escapes every compact
set.  Testing against compactly supported $L^{p'}$ vector fields proves
weak convergence to zero, and density gives the assertion for every
$v\in\Dp(\Hhalf)$.  Finally, if the expression in
\eqref{eq:abstract-parameter-divergence} is bounded, then $a_k$ is
bounded above and away from zero and $b_k$ is bounded, which proves the
last assertion.
\end{proof}

\subsection{Proof of global compactness}

\begin{proof}[Proof of Theorem~\ref{thm:global-compactness}]
By \eqref{eq:positive-signed-comparison}, the assumptions imply
$P_+(u_k)\to0$ in $W^{-1,p'}(\Hhalf)$.  Passing to a subsequence, write
\[
 u_k\rightharpoonup U^0\quad\text{in }\Dp(\Hhalf).
\]
Lemma~\ref{lem:local-compactness} shows that $P_+(U^0)=0$ and that
subtraction of $U^0$ splits the norm, boundary mass, energy, and
residual.  We initialize the profile list as follows.  If $U^0\ne0$,
put
\[
 q:=1,\qquad V^1:=U^0,\qquad
 \lambda_1^{(k)}:=1,\quad\xi_1^{(k)}:=0,\quad
 g_1^{(k)}:=G_{1,0},\qquad
 r_k^1:=u_k-V^1.
\]
If $U^0=0$, put $q:=0$ and $r_k^0:=u_k$.  In both cases
\begin{equation}
 r_k^q\rightharpoonup0,\qquad P_+(r_k^q)\longrightarrow0,
 \label{eq:initial-remainder-properties}
\end{equation}
and, when $q=1$, $(g_1^{(k)})^{-1}r_k^1\rightharpoonup0$.

We now construct the remaining profiles, beginning with $j=q+1$.
Suppose that profiles up to index $j-1$ have been constructed and that
$r_k^{j-1}$ is bounded with $P_+(r_k^{j-1})\to0$.  If
$(r_k^{j-1})_+\to0$ in $L^{p_*}(\bdH)$, the construction stops.
Otherwise Lemma~\ref{lem:profile-extraction} gives a transformation
\[
 g_j^{(k)}:=G_{\lambda_j^{(k)},\xi_j^{(k)}}
\]
and a nonzero profile $V^j$ such that
\begin{equation}
 z_k^j:=(g_j^{(k)})^{-1}r_k^{j-1}\rightharpoonup V^j,
 \qquad
 r_k^j:=r_k^{j-1}-g_j^{(k)}V^j.
 \label{eq:inductive-profile-definition}
\end{equation}
In particular,
\begin{equation}
 (g_j^{(k)})^{-1}r_k^j=z_k^j-V^j\rightharpoonup0.
 \label{eq:remainder-zero-in-own-coordinates}
\end{equation}
The norm and residual splittings in
Lemma~\ref{lem:profile-extraction} show that $r_k^j$ is bounded and
$P_+(r_k^j)\to0$.

The construction stops after finitely many steps.  Indeed, iteration
of the norm splitting gives, at every constructed index $j$,
\begin{equation}
 \|\nabla u_k\|_p^p
 =\sum_{\ell=1}^j\|\nabla V^\ell\|_p^p
 +\|\nabla r_k^j\|_p^p+o(1).
 \label{eq:iterated-norm-splitting}
\end{equation}
Every profile contributes at least $A_{\mathrm T}$ by
Lemma~\ref{lem:profile-gap}, whereas $(u_k)$ is bounded.  Consequently
there is a final index $m$ and
\[
 \|(r_k^m)_+\|_{L^{p_*}(\bdH)}\longrightarrow0.
\]
Testing $P_+(r_k^m)$ by $r_k^m$ gives
\[
 \|\nabla r_k^m\|_p^p
 =\int_{\bdH}(r_k^m)_+^{p_*}\,dy+o(1),
\]
and hence $r_k^m\to0$ strongly in $\Dp(\Hhalf)$.

By construction,
\[
 u_k=\sum_{j=1}^m g_j^{(k)}V^j+r_k^m,
\]
so this is already the asserted strong decomposition before
classification.

We next prove the parameter orthogonality, separately from the
extraction argument.  We use induction on the later profile index $j$.
Suppose that orthogonality has already been proved for all pairs whose
second index is less than $j$, and fix $i<j$.  From the successive
definitions of the remainders,
\begin{equation}
 (g_i^{(k)})^{-1}r_k^{j-1}
 =(g_i^{(k)})^{-1}r_k^i
 -\sum_{\ell=i+1}^{j-1}(g_i^{(k)})^{-1}g_\ell^{(k)}V^\ell.
 \label{eq:earlier-coordinate-remainder-expansion}
\end{equation}
The first term tends weakly to zero by
\eqref{eq:remainder-zero-in-own-coordinates} when $V^i$ was extracted,
and by \eqref{eq:initial-remainder-properties} when $V^i=U^0$ is the
fixed weak limit.  Each term in the sum tends weakly to zero by the
inductive orthogonality and Lemma~\ref{lem:dislocation-facts}.  Thus
\begin{equation}
 (g_i^{(k)})^{-1}r_k^{j-1}\rightharpoonup0.
 \label{eq:later-remainder-zero-in-earlier-coordinates}
\end{equation}
On the other hand, the group law gives
\begin{equation}
 (g_i^{(k)})^{-1}g_j^{(k)}
 =G_{a_{ij,k},b_{ij,k}},
 \qquad
 a_{ij,k}:=\frac{\lambda_j^{(k)}}{\lambda_i^{(k)}},
 \quad
 b_{ij,k}:=\frac{\xi_j^{(k)}-\xi_i^{(k)}}{\lambda_i^{(k)}}.
 \label{eq:relative-dislocation-parameters}
\end{equation}
If the desired orthogonality failed, then along a subsequence
\begin{align*}
 a_{ij,k}+\frac1{a_{ij,k}}+\frac{|b_{ij,k}|^2}{a_{ij,k}}
 &=\frac{\lambda_j^{(k)}}{\lambda_i^{(k)}}
 +\frac{\lambda_i^{(k)}}{\lambda_j^{(k)}}
 +\frac{|\xi_j^{(k)}-\xi_i^{(k)}|^2}
 {\lambda_i^{(k)}\lambda_j^{(k)}}
\end{align*}
would be bounded.  Lemma~\ref{lem:dislocation-facts} would then give
a further subsequence for which
\[
 (g_i^{(k)})^{-1}g_j^{(k)}\longrightarrow h.
\]
Since
\[
 z_k^j=(g_j^{(k)})^{-1}r_k^{j-1}\rightharpoonup V^j,
\]
part~1 of Lemma~\ref{lem:dislocation-facts} would give
\[
 (g_i^{(k)})^{-1}r_k^{j-1}
 =\bigl((g_i^{(k)})^{-1}g_j^{(k)}\bigr)z_k^j
 \rightharpoonup hV^j\ne0,
\]
contradicting \eqref{eq:later-remainder-zero-in-earlier-coordinates}.
This proves \eqref{eq:global-parameter-orthogonality}.

Iterating the initial splittings
\eqref{eq:bulk-splitting}--\eqref{eq:functional-splitting} and the
profile splittings
\eqref{eq:profile-norm-splitting}--\eqref{eq:profile-energy-splitting}, and using the strong convergence
of the last remainder, gives
\begin{align*}
 \int_{\Hhalf}|\nabla u_k|^p\,dx
 &\longrightarrow\sum_{j=1}^m
 \int_{\Hhalf}|\nabla V^j|^p\,dx,\\
 J_+(u_k)&\longrightarrow\sum_{j=1}^mJ_+(V^j).
\end{align*}
Finally, the classification theorem \cite[Theorem~1.1]{Zhou2024}
identifies each $V^j$ with $W[\mu_j,\eta_j]$ for fixed $\mu_j>0$ and
$\eta_j\in\R^{n-1}$.  By \eqref{eq:dislocation-group-law},
\[
 g_j^{(k)}V^j
 =W[\lambda_j^{(k)}\mu_j,
 \xi_j^{(k)}+\lambda_j^{(k)}\eta_j].
\]
This absorption of fixed parameters preserves orthogonality: if
$h_j:=G_{\mu_j,\eta_j}$, then the new relative dislocation is
$h_i^{-1}\bigl((g_i^{(k)})^{-1}g_j^{(k)}\bigr)h_j$, and multiplication
by fixed group elements preserves escape from compact subsets of the
parameter group.  Renaming the final parameters proves
\eqref{eq:global-decomposition}--\eqref{eq:global-energy-quantization}.
\end{proof}

Theorem~\ref{thm:global-compactness} yields
Corollary~\ref{thm:trace-nu-bubble-compactness}.  Combining
Theorem~\ref{thm:trace-critical-stability} with the case $\nu=1$ of
Corollary~\ref{thm:trace-nu-bubble-compactness} then gives
Corollary~\ref{cor:global-one-bubble-critical-stability}.
\section{Use of Generative-AI Tools Declaration}
\textbf{Statement:} During the preparation of this work, the authors used ChatGPT (OpenAI) to assist in identifying potential grammatical and logical errors in the manuscript. After using this tool, the authors reviewed and edited the content as needed and take full responsibility for the content of the published article.
\section{Acknowledgments}
  We would like to thank Prof. Xi-Nan Ma for his advanced guidance. This work was supported by National Natural Science Foundation of China [grant number: 2025YFA1017601].

\end{document}